\documentclass[twoside, 12pt, reqno]{amsart}

\usepackage[foot]{amsaddr}

\makeatletter
\def\subsubsection{\@startsection{subsubsection}{3}%
\z@{.5\linespacing\@plus.7\linespacing}{-.5em}%
{\normalfont\bfseries}}
\makeatother

\usepackage{xcolor}

\usepackage[numbers]{natbib}
\usepackage{amsfonts}
\usepackage{amsmath}
\usepackage{amsthm}
\usepackage{amssymb}
\usepackage{color}
\usepackage{graphicx}
\usepackage{hyperref}
\hypersetup{
    colorlinks=true, 
    linktoc=all,     
    linkcolor=blue,  
}

\usepackage{comment}
 \usepackage[top=0.75in, bottom=0.75in, left=0.75in, right=0.75in]{geometry}

\usepackage{algorithm}     
\usepackage{algpseudocode} 

\usepackage[shortlabels]{enumitem}
\usepackage{mathrsfs}  
\usepackage{mathtools}
\usepackage{xcolor}

\usepackage{enumitem}

\usepackage{multirow}

\usepackage{dsfont}

\numberwithin{equation}{section}

\theoremstyle{plain}

\newtheorem{theorem}{Theorem}[section]
\newtheorem{lemma}{Lemma}[section]
\newtheorem{prop}{Proposition}[section]
 \newtheorem{corollary}{Corollary}[section]
\newtheorem{assumption}{H.\!\!}

\theoremstyle{definition}

\newtheorem{definition}[theorem]{Definition}

\theoremstyle{remark}
\newtheorem{remark}[theorem]{Remark}

\newtheorem{example}{Example}[section]

\numberwithin{table}{section}

\newcommand{\diag}{\operatorname{diag}}

\newcommand{\E}{\mathbb{E}}
\newcommand{\R}{\mathbb{R}}

\newcommand{\Tr}{\operatorname{tr}}

\usepackage{bbm}

\def\cA{\mathcal{A}}

\def\cF{\mathcal{F}}
\def\cG{\mathcal{G}}
\def\cH{\mathcal{H}}

\def\cO{\mathcal{O}}
\def\cP{\mathcal{P}}

\def\cS{\mathcal{S}}

\def\d{{\mathrm{d}}}

\def\bu{{\textbf{\textit{u}}}}

\def\bz{{\textbf{\textit{z}}}}

\def\bX{{\textbf{X}}}
\def\bY{{\textbf{Y}}}

\def\var{{\text{Var}}}

\def\sE{{\mathbb{E}}}
\def\sF{{\mathbb{F}}}

\def\sN{{\mathbb{N}}}
\def\sP{\mathbb{P}}

\def\sR{{\mathbb R}}
\def\sS{{\mathbb{S}}}
\def\sV{{\mathbb{V}}}

\DeclareMathAlphabet{\mymathbb}{U}{bbold}{m}{n} 

\def \Ck2{\begin{pmatrix} C_k & 0 \\ 0 & C_k \end{pmatrix}}

\def\B2{{\begin{pmatrix}
       \Tilde{B} & 0 \\   0 &  \Tilde{B}  \end{pmatrix}}}

\def\A2{{\begin{pmatrix}
       A & 0 \\   0 &  A  \end{pmatrix}}}

\def\ra{{{\rightarrow}}}

\colorlet{blue}{black}

\title[$\alpha$-potential distributed games with jumps]{Distributed games with jumps: an $\alpha$-potential game Approach}

\author{Xin Guo$^\dagger$}
\email{xinguo@berkeley.edu}
\author{Xinyu Li$^{\S }$}
\email{xinyu.li@maths.ox.ac.uk}
\author{Yufei Zhang$^\ddagger$}
\email{yufei.zhang@imperial.ac.uk}

\address{$^\dagger$ {Department of Industrial Engineering and Operations Research, UC Berkeley,  CA, USA}}
\address{$^\S$ {Mathematical Institute, University of Oxford, Oxford, UK} }
\address{$^\ddagger${Department of Mathematics, Imperial College London, London, UK}}



\begin{document}
\maketitle 

\begin{abstract}
Motivated by game-theoretic  models of crowd motion dynamics, this paper analyzes a broad class of distributed games with jump diffusions within the recently developed $\alpha$-potential game  framework. We demonstrate that analyzing the $\alpha$-Nash equilibria  reduces to solving a finite-dimensional control problem.
Beyond the  viscosity and verification characterizations for the general games, we examine explicitly and in detail how spatial population distributions and interaction rules influence the structure of $\alpha$-Nash equilibria in these distributed settings. 

For crowd motion network games, 
we show  that $\alpha = 0$ for all symmetric interaction networks, and {\color{blue}for asymmetric networks, we quantify the precise polynomial and logarithmic decays of $\alpha$ in terms of the number of players, the degree of the network, and the decay rate of interaction asymmetry.}
{\color{blue}
We also exploit the $\alpha$-potential game framework to analyze an $N$-player portfolio selection game under a mean–variance criterion.  We show that this   portfolio game constitutes a potential game  and explicitly construct its Nash equilibrium.
Our analysis allows for heterogeneous preference parameters, going beyond the mean-field interactions considered in the existing game literature.
 
}

Our theoretical results are supported by numerical implementations using policy gradient-based algorithms,  demonstrating the computational advantages of the $\alpha$-potential game framework in computing Nash equilibria for general dynamic games.

\end{abstract}

\section{Introduction}
 \label{sec:intro}

\subsection*{Motivating example and distributed games.}
Consider the following motion planning game  \cite{lachapelle2011mean, 
aurell2018mean, 
santambrogio2021cucker,
carmona2023synchronization}, where a group of $N$ players each controls or  chooses their 
preferred route  to reach their respective destinations; their paths are impacted by the spatial distribution of  the population and their interactions.  In this game, each player aims to find the optimal path according  to her cost functional consisting of terminal costs and the running costs 
which depend on the controls and the path to her  destination. This game can be modeled   as the following stochastic differential game. 
For each player $i \in [N]\coloneqq \{1, 2, \cdots, N\}$, 
given her control process $u_i$, 
her state process $X^{u_i}_i$,
representing the player's position,
is governed by  the following controlled jump-diffusion process:
\begin{equation}
\label{eq:dynamics_ex_intro}
\begin{aligned}
      \d X_{i,t}
      &= b_i(t) u_{i,t}  \d t + \sigma_i(t) \d W_t + 
      \sum_{j=1}^m\int_{\sR_0^p}\gamma_{ij}(t, z) \widetilde {\eta}_{j}(\d t, \d z),
      \quad 
      t\in (0,T];\quad X_{i,0} = x_i\in \sR^d,
\end{aligned}
\end{equation}
where  $ b_i : [0,T] \ra \sR^{d\times k}$, $\sigma_i : [0,T] \ra \sR^{d  \times n}$ and $\gamma_{ij} : [0,T] \times \sR^p \ra \sR^{d }$
are   measurable  functions such that 
\eqref{eq:dynamics_ex_intro}
admits a unique  strong   solution  on an appropriate probability space which supports the $n$-dimensional Brownian motion $W$ and the jump processes $(\tilde \eta_j)_{j=1}^m$.
Given a joint control profile  $\bu =(u_i)_{i\in [N]}$ from an admissible set,  
each player $i$ aims to minimize  over her admissible controls  an objective function of the form
\begin{align}
\label{eq: cost_ex_intro}
J_i(\bu)\coloneqq\sE\left[\int_0^T \left( 
 \ell_i( u_{i,t}) + \frac{1}{N -1 } \sum_{j=1, j\ne i}^N q_{ij}  K( X^{u_i}_{i,t} -X^{u_j}_{j,t}) \right)\,\d t+ g_i (X^{u_i}_{i,T})\right],
\end{align}
where $\ell_i$ is the cost of control and $g_i$ is the terminal cost. The kernel function  $K$  can be specified to model self-organizing behavior such as  flocking, or aversion behavior, with adjustment of the interaction intensity by $q_{ij} \ge 0$.

The above crowd motion  game 
is a special class of  stochastic differential games which we name as  distributed games; See Section \ref{sec:setup} for the detailed formulation of these games. The term ``distributed'' refers to the characteristics of the game where each player's dynamics evolve according to a controlled stochastic process that depends only on her own control, while her objective function may depend on the joint state and control profiles of all players  (see also Remark \ref{rmk:distinction}).
 Such a framework has been used  in a variety of applications where agents interact through their objectives but evolve independently in state, including distributed control of multi-agent systems and trajectory planning \cite{ kavuncu2021potential, sun2023distributed, sun2024imagined, aghajani2015formation}, transportation and routing \cite{blum2010routing, colombo2012efficient, krichene2015convergence},  as well as in energy markets and smart grids \cite{ma2011decentralized, srikantha2016resilient, ramchurn2011agent,paccagnan2016aggregative,tushar2020peer,narasimha2022multi}.
{\color{blue}
The jump processes 
$(\tilde \eta_j)_{j=1}^m$ in the state dynamics \eqref{eq:dynamics_ex_intro}  capture sudden stochastic shocks and can be used to model neuronal spiking activity in neuroscience \cite{musila1991generalized,sirovich2014cooperative}, as well as stochastic demand shocks that introduce discontinuities in production and inventory management \cite{liu2018optimal}.
 
}

In general, deriving Nash equilibria for this type of game is analytically challenging, as the interaction kernel $K$ in \eqref{eq: cost_ex_intro} is typically non-convex, which precludes the use of standard tools such as the stochastic  maximum principle. An exception arises in the special case of mean field games, under the assumptions that players are homogeneous and interact weakly through empirical measures, and the number of players $N \to \infty$, 
see for instance, 
\cite{lachapelle2011mean, 
aurell2018mean, 
santambrogio2021cucker,
carmona2023synchronization}.

Meanwhile, the recently introduced $\alpha$-potential game framework has shown significant promise for analyzing and solving general dynamic games, both from theoretical and algorithmic perspectives \cite{guo2023markov, guo2025towards, guo2024alpha,  guo2025bsde, kalaria2024alpha, maheshwari2024convergence}. 
The $\alpha$-potential game framework directly addresses finite-player games, unlike the conventional mean-field game approach, which relies on weak interactions among players or considers the limit as the number of players $N \to \infty$.
Specifically, \cite{guo2023markov,guo2024alpha} demonstrate that computing a Nash equilibrium can be reformulated as an optimization problem involving a single $\alpha$-potential function and the analysis of the parameter $\alpha$. 
Furthermore, \cite{guo2024alpha} constructs a specific $\alpha$-potential function for general stochastic differential games and optimizes it by reformulating the problem as a conditional McKean–Vlasov control problem. {\color{blue} As a result, \cite{guo2024alpha} obtains an $\alpha$-Nash equilibrium via a Hamilton–Jacobi–Bellman (HJB) equation on the infinite-dimensional space of probability measures. However, beyond the LQ setting, analyzing this infinite-dimensional control problem and computing the $\alpha$-Nash equilibrium becomes highly challenging.
}

\subsection*{Our approach and our work.}

In this paper, we develop the $\alpha$-potential game framework for distributed games with controlled jump-diffusion dynamics.

\begin{itemize}
 \item  Within this framework, 
we show   that the task of finding an $\alpha$-Nash equilibrium can be further reduced to solving a finite-dimensional control problem (Theorem~\ref{thm:potential_diff}),
{\color{blue} 
thereby avoiding the infinite-dimensional McKean–Vlasov control formulation given in \cite{guo2024alpha}. A key step is to obtain a separation principle for the controlled state  processes (Lemma \ref{lemma:state_decomposition}).
As a consequence, we obtain new Markovian characterizations of $\alpha$-Nash equilibria via a finite-dimensional HJB equation, utilizing both verification methods and viscosity solution theory
(Theorems \ref{thm:verification_full}
and \ref{thm:hjb_full_viscosity}).
}
 
\item

We   precisely  characterize  how the parameter $\alpha$ depends on the underlying game structure, {\color{blue} including the network governing player interactions, extending beyond the results in \cite{guo2024alpha}.} In particular,  for games of the form
\eqref{eq:dynamics_ex_intro}-\eqref{eq: cost_ex_intro}, we explicitly demonstrate how the resulting $\alpha$-Nash equilibria are shaped by
 the choice of the kernel function $K$  and the   intensity and asymmetry  of the interaction weights  $(q_{ij})_{i,j\in [N]}$ in the   payoffs  
(Theorem  \ref{thm:crowd_alpha_general} and Corollary \ref{cor:crowd_interaction_alpha}).
We prove that $\alpha = 0$ for all symmetric interaction networks, and {\color{blue}for asymmetric networks, we quantify the precise polynomial and logarithmic decays of $\alpha$ in terms of the number of players, the degree of the network, and the decay rate of interaction asymmetry.}

\item 
Leveraging the finite-dimensional control formulation of the $\alpha$-potential function, 
we  develop an efficient policy gradient algorithm 
 (Algorithm \ref{alg:PG})
to compute an $\alpha$-Nash equilibrium. Through numerical experiments on crowd motion games, we showcase distinct emergent trajectories in both flocking and aversion dynamics.

\item 
{\color{blue}
We exploit the $\alpha$-potential game framework to analyze an $N$-player portfolio selection game under a mean–variance criterion. Players trade in a financial market with multiple risky assets, whose prices are driven by Brownian motions and Poisson random measures. Each player aims to maximize their total wealth relative to a weighted average of the other players’ wealth while penalizing variance. We show that this   portfolio game constitutes a potential game  and explicitly construct the corresponding Nash equilibrium (Theorem \ref{thm:potential_maximizer}). Our analysis allows for  heterogeneous preference parameters, going beyond the mean-field interactions considered in the existing game literature.
 
}

\end{itemize}

\subsection*{Notation}

 Let  
$T > 0$. For each 
   measurable function $\phi:[0,T]\to \sR^n$,
we define its $L^2$-norm $\|\phi\|_{L^2}=\left(\int_0^T |\phi(s)|^2\d s\right )^{1/2}$ with $|\cdot|$ being the Euclidean norm, and 
for each 
$\phi:U\to \sR^{m\times n}$ defined on a set $U$,
we define its sup-norm
$\|\phi\|_{L^\infty}=\sup_{u\in U}\|\phi(u)\|_{\textrm{sp}}$, with $\|\cdot\|_{\textrm{sp}}$
being the spectral  norm of a matrix. 

For each filtered probability space 
$(\Omega, \cF, \sF, \sP)$
  and Euclidean space $(E, |\cdot |)$,  
 we denote by 
 $\cS^2( E)$    the space of $E$-valued $\sF$-progressively measurable processes $X:\Omega\times [0,T]\to E$  satisfying $\|X\|_{\cS^2(E)}=\sE[\sup_{s\in [0,T]}|X_s|^2]^{1/2}<\infty$,
 and 
 by   $\cH^2(E)$    the space of $E$-valued $\sF$-progressively measurable processes $X:\Omega\times [0,T]\to E$  satisfying $\|X\|_{\cH^2(E)}=\sE[\int_0^T|X_s|^2\d s]^{1/2}<\infty$.
 With a slight abuse of notation, for any $m,n\in \sN$,
 we identify the product spaces $ \cS^2(\sR^n)^m$ and  $\cH^2(\sR^n)^m$ with 
 $\cS^2(\sR^{mn})$ and 
  $\cH^2(\sR^{mn})$, respectively.

\section{Distributed Games and Their Nash Equilibria} \label{sec:setup}

This section introduces a class of stochastic differential games, referred to as \emph{distributed games},
in which each player's dynamics evolve according to a drift-controlled jump-diffusion process that depends only on their own control, while their objective function may depend on the joint state and control profiles of all players.  We next present preliminary results for applying the $\alpha$-potential game framework developed in \cite{guo2024alpha} to compute approximate Nash equilibria for such games.


\subsection{Mathematical Setup}
\label{sec:math_setup}
Let $T>0$ be {\color{blue} a} given terminal time, and 
    $N, d, n, m, p \in \mathbb{N}$.
    Let  
    $ (\Omega, \mathcal{F}, \mathbb{P}) $ be a complete probability space which supports the following three mutually independent processes: 
     a family of   square integrable $d$-dimensional random variables $(\xi_i)_{i=1}^N$, 
    an $n$-dimensional Brownian motion $  W =(W_i)_{i=1}^n$, and a family of 
 independent   Poisson random measures  
  $  \eta= (\eta_i)_{i=1}^m  $ on $[0,T]\times \mathbb{R}^p_0$, 
  where $\mathbb{R}^p_0\coloneqq \sR^p\setminus \{0\}  $ is equipped with its Borel $ \sigma $-algebra.
  The random variables $(\xi_i)_{i=1}^N$ represents the 
   initial conditions of the  system states,  and the processes $W$ and $\tilde \eta$ represent the underlying system  noises. We assume that 
  each  
  $\eta_i $ has a compensator $  \nu_i(\d z )\,\d t $,
  with     $ \nu_i $ being  a $ \sigma $-finite measure on $  \mathbb{R}^p_0 $ satisfying
$ 
\int_{\mathbb{R}^p_0} \min(1, |z |^2) \, \nu_i(\d z) < \infty,
$
and define  $ \widetilde{\eta}_i(\d t, \d z) = \eta_i(\d t, \d z) - \nu_i(\d z)\,\d t $ as  the compensated Poisson random measure of $\eta_i$.
Let   $ \mathbb{F} = (\mathcal{F}_t)_{t \in [0,T]} $ be the filtration generated by 
$(\xi_i)_{i=1}^N$, 
$ W $ and $ \eta $, augmented with the $ \mathbb{P} $-null sets.

We consider a stochastic differential game involving $N$ players,  each employing open-loop control strategies defined as follows.\footnotemark  
For each $i\in [N]\coloneqq \{1,\ldots, N\}$, 
let $A_i\subset \sR^k$ be   player $i$'s    action set, 
and 
let 
$\cA_i$ be  the set of  player $i$'s  admissible controls defined by 
\begin{equation}
\label{eq:admissible_control}
    \cA_i\coloneqq 
\{u:\Omega\times [0,T]\to A_i
\mid   
u \in \cH^2(\sR^k) 
\}.
\end{equation}
  Let $A=\prod_{i\in [N]}A_i$
be the set of  joint action profiles of all players and 
$\cA=\prod_{i\in [N]}\cA_i$ be the joint control profiles.
For each $i\in [N]$,
we denote by  
$\cA_{-i}
  \coloneqq 
  \prod_{j\in [N]\setminus\{i\}} 
 \cA_j$ the set of control profiles of all players except player $i$, and 
 by $\bu=(u_i)_{i\in [N]}$  and $\bu_{-i}=(u_j)_{j\in [N]\setminus\{i\}}$ a generic element
of $\cA$ and $\cA_{-i}$,  respectively. 

\footnotetext{
{\color{blue} Open-loop controls refer to strategies that are non-anticipative functions of the system noise and the initial states (see e.g., \cite[Chapter 2]{carmona2018probabilistic}).
We will prove in Section \ref{sec:optimize_potential_function} that the resulting Nash equilibria admit a feedback representation. 
The $\alpha$-potential game framework can also be applied to games with closed-loop controls, where strategies are induced by policies that depend directly on the current system state \cite{guo2023markov,di2025alpha,plank2026learning}.
In this setting, 
the characterization of $\alpha$ 
 additionally depends on the regularity and structural properties of the closed-loop policies, and is therefore typically more involved than in the open-loop case.
}
} 

Given the control sets $(\cA_i)_{i\in [N]}$, each player influences their evolution by controlling the drift of a jump-diffusion process.  
More precisely, 
for each $\bu=(u_i)_{i\in [N]}\in \cA$,
consider    
the following state  dynamics: for all 
   $i\in [N]$, 
\begin{equation}\label{eq:X_state}
\begin{aligned}
      \d X_{i,t}
      &= b_i(t) u_{i,t}  \d t + \sigma_i(t) \d W_t + 
      \sum_{j=1}^m\int_{\sR_0^p}\gamma_{ij}(t, z) \widetilde {\eta}_{j}(\d t, \d z),
      \quad 
      t\in (0,T];\quad X_{i,0} = \xi_i,
\end{aligned}
\end{equation}
where  $ b_i : [0,T] \ra \sR^{d\times k}$, $\sigma_i : [0,T] \ra \sR^{d  \times n}$ and $\gamma_{ij} : [0,T] \times \sR^p \ra \sR^{d }$
are   measurable  functions such that 
\eqref{eq:X_state}
admits a unique  strong   solution  $\bX^{\bu} =  (X_i^{u_i})_{i\in [N]} \in \cS^2(\sR^{dN})$; see   (H.\ref{assum:regularity_general}) for the precise conditions.
Player  $i$ determines their optimal strategy by minimizing the  following objective function 
$J_i:\cA\to \sR$:
\begin{equation}\label{eqn:objective_function}
\inf_{u_i\in \cA_i}
J_i(\boldsymbol{u}),
\quad 
\textnormal{with}
\quad J_i(\boldsymbol{u})\coloneqq \mathbb{E}\left[\int_0^T f_i\left(t, \mathbf{X}_t^{\boldsymbol{u}}, \boldsymbol{u}_t\right) \mathrm{d} t+g_i\left(\mathbf{X}_T^{\boldsymbol{u}}\right)\right],
 \end{equation}
 where  the running cost  $f_i:[0, T] \times \mathbb{R}^{d N } \times \mathbb{R}^{k N } \rightarrow \mathbb{R}$ and 
 the terminal cost $g_i: \mathbb{R}^{d N } \rightarrow \mathbb{R}$ are given measurable functions. 

We denote by $\mathcal{G}$ the game defined by \eqref{eq:X_state}–\eqref{eqn:objective_function}, and refer to it as a \emph{distributed game}, as each player’s state is governed solely by their own control. The game $\mathcal{G}$  includes as a special case  
 the game-theoretic models for crowd motion dynamics that 
 will be analyzed in detail in Section \ref{sec:crowd_motion}. 
 In these models, the state process represents each player’s position and/or velocity, and the cost function captures each player’s target region, energy expenditure for traveling, and preferred route, which depends on the spatial distribution of the population.
 See Section \ref{sec:crowd_motion} for more details. 

\begin{remark}\label{rmk:distinction}

Note the distinction between distributed games and distributed controls, the latter of which typically assumes that  all players’ states are independent and that each player's control depends only on their own state (see, e.g., \cite{jackson2023approximately}). In contrast, in distributed games players’ states and control processes can be correlated due to shared sources of randomness, such as correlated initial states and common components in the Brownian motions or the Poisson random measures.
 
\end{remark}

\begin{remark}

Note that \eqref{eq:X_state} can accommodate linear dependence on the state variable in the drift via a simple change of variables.
Indeed, suppose
that 
for each $\bu \in \cA$,
player $i$'s state dynamics 
$\Tilde{X}_{i}^{u_i}$ satisfies for all $t\in [0,T]$,
\begin{align}
\label{eq:state_linear}
\d \Tilde{X}_{i,t}  = \left( \Tilde{a}_i(t)\, \Tilde{X}_{i,t}  + \Tilde{b}_i(t)\, u_{i,t} \right) \d t + \Tilde{\sigma}_i(t)\, \d W_t  + 
      \sum_{j=1}^m\int_{\sR_0^p}
      \Tilde{\gamma}_{ij}(t, z) \widetilde {\eta}_{j}(\d t, \d z),
      \quad \Tilde{X}_{i,0}^{u_i}= \xi_i,
\end{align}
where 
$\Tilde{a}_i, \Tilde{b}_i, \Tilde{\sigma}_i $
and $\Tilde{\gamma}_{ij}$ are  given measurable functions.  Then by considering   
$$
X_{i,t}^{u_i} \coloneqq  A_i(t) \Tilde{X}_{i,t}^{u_i},
\quad t\in [0,T],
$$ 
with $\frac{\d }{\d t}A_i(t)=-A_i(t)\tilde a_i(t)$
and $A_i(0)=I_d$, 
the state dynamics
\eqref{eq:state_linear}
can be transformed into the simpler form given in \eqref{eq:X_state}, with the state coefficients and cost functions adjusted by certain deterministic factors.
A special case of \eqref{eq:state_linear} is the following controlled kinetic equation
(see e.g, \cite{nourian2011mean,santambrogio2021cucker}):
for all $t\in [0,T]$,
 \begin{equation*} 
  \left\{
   \begin{aligned}
\d x_{i,t} & = v_{i,t}  \d t, \quad x_{i,0}=x_i,
 \\ 
\d v_{i,t} &= u_{i,t} \d t
 + \Tilde{\sigma}_i(t)\, \d W_t  + 
      \sum_{j=1}^m\int_{\sR_0^p}
      \Tilde{\gamma}_{ij}(t, z) \widetilde {\eta}_{j}(\d t, \d z),
      \quad v_{i,0}=v_i,
\end{aligned}\right.
 \end{equation*}
  where  $x_{i,t}$ and  $v_{i,t}$  denote  player $i$'s position and   velocity  at time $t$, respectively.

\end{remark}

Throughout this paper,
we impose the following standing regularity condition  on the coefficients of \eqref{eq:X_state}-\eqref{eqn:objective_function}. 
\begin{assumption}
\label{assum:regularity_general}
    For all $i,j\in [N]$, $A_i\subset \sR^k$ is   convex and  $0\in A_i$, and 
    $b_i, \sigma_i, \gamma_{ij}, f_i$ and $g_i$ are measurable functions satisfying   the following conditions:
    
\begin{enumerate}[(1)]
\item 
\label{item:state_coefficient}
$b_i$ and $\sigma_i$ are square integrable, and $\sup_{(t,z)\in [0,T]\times \sR^p_0} {|\gamma_{ij}(t,z)|}/{\min(1,  |z|)}<\infty$.
\item 
\label{item:f_g_general}
For all  $t\in [0,T]$,
    $\sR^{dN}\times \sR^{kN}\ni (x,a)\mapsto ( f_i(t,x,a), g_i(x))\in \sR\times \sR$ is twice continuously differentiable, 
$[0,T]\ni t\mapsto \big( f_i(t,0,0),  \partial_{(x,a)}f_i (t,0,0) \big)\in \sR\times \sR^{(d+k)N}$ is bounded,
and 
the second-order derivatives
$\partial^2_{xx}f_i$,
$\partial^2_{xa}f_i$,
$\partial^2_{aa}f_i$,
and 
$\partial^2_{xx}g_i$
are bounded (uniformly in $(t,x,a)$).
\end{enumerate}
 
\end{assumption}

Under Assumption (H.\ref{assum:regularity_general}),
for each $\bu\in \cH^2(\sR^{kN})$, \eqref{eq:X_state} admits a unique strong solution $\bX^{\bu}  \in \cS^2(\sR^{dN})$  (see \cite[Theorem 3.1]{kunita2004stochastic}), and \eqref{eqn:objective_function} is well defined. 
For ease of exposition, we assume that the action set contains $0$, but similar analyses can be extended to a  non-empty convex action set (see e.g., \cite{guo2024alpha, guo2025towards}).

\subsection{NEs and $\alpha$-potential function}

We aim to characterize the rational behavior of the players in the distributed game $\mathcal{G}$.
To this end, we first recall the notion of an $\varepsilon$-Nash equilibrium, defined as a joint control profile in which no player can improve their performance by more than $\varepsilon$ through any unilateral deviation. The precise definition is given below.

\begin{definition}
\label{def:epsilon_NE}
     For any $\varepsilon\ge 0$, a control profile $\bar \bu = (\bar u_i )_{i\in [N]} \in \cA$ is an $\varepsilon$-Nash equilibrium of the game $\cG$ if
$ J_i\left( \bar \bu \right) \leq J_i\left(\left(u_i , \bar u_{-i}\right)\right) + \varepsilon$, for all  $i \in [N],  u_i  \in \cA_i $.
\end{definition}

To analyze and compute an approximate NE of the game $\mathcal{G}$,
we employ the $\alpha$-potential game framework introduced in \cite{guo2024alpha}. 
\begin{definition}
Consider the game $\mathcal{G}$ in \eqref{eq:X_state}-\eqref{eqn:objective_function}. We say $\cG$ is an $\alpha$-potential game 
for $\alpha\ge 0$
if there exists 
a function $\Phi: \cA  \to \R$ such that 
for all $ i \in [N]$, $u_i, u_i^{\prime} \in \cA_i$ and $u_{-i} \in \cA_{-i}$,
\begin{equation}
\label{eq:alpha_v_phi}
|J_i\left(\left(u_i^{\prime}, u_{-i}\right)\right)-J_i\left(\left(u_i, u_{-i}\right)\right) - \left( \Phi\left(\left(u_i^{\prime}, u_{-i}\right)\right)-\Phi\left(\left(u_i, u_{-i}\right)\right) \right) | \leq \alpha.    
\end{equation}
Such a  function $\Phi$ is   called an $\alpha$-potential function for $\cG$.
In the case where $\alpha=0$,
we simply call the game $\cG$ a potential game and $\Phi$ a potential function for $\cG$.
\end{definition}
 The main advantage of this framework is that, once such an $\alpha$-potential function $\Phi$ is constructed, 
 finding approximate NEs reduces to solving a single optimization problem: minimizing $\Phi$ over $\cA$. This connection is made precise in the following lemma.

\begin{lemma}
[{\cite[Proposition 2.1]{guo2024alpha}}]

\label{lemma:a_potential_game}

Let    $\Phi: \cA  \to \sR$  be an $\alpha$-potential function of the game $\mathcal{G}$. For each    $\varepsilon\ge 0$, if  $\bar \bu  \in \cA $ satisfies   $\Phi(\bar \bu )\le \inf_{\bu \in \cA }\Phi(\bu)+\varepsilon$, then $\bar \bu  $ is an $( \alpha+\varepsilon) $-NE of the game $\mathcal{G}$.
 
\end{lemma}

 As shown in \cite{guo2024alpha}, one can construct an $\alpha$-potential function for a stochastic differential game using  the linear derivatives of each player’s objective function.
For each $i,j\in [N]$,
we say 
$f:\cA\to \sR$ has a linear derivative in  $\cA_j$
if there exists 
a function  
$\frac{\delta f}{\delta u_j}:\cH^2(\sR^{kN})\times  \cH^2(\sR^k)\to \sR$ such that  for all $\bu \in \cA$, 
$\frac{\delta f}{\delta u_j} (\bu ; \cdot)$ is linear    and 
\begin{equation*}
\lim _{\varepsilon \searrow 0} \frac{f\left(\left(u_j+\varepsilon\left(u_j^{\prime}-u_j\right), u_{-j}\right)\right)-f(\bu)}{\varepsilon}=\frac{\delta f}{\delta u_j}\left(\bu ; u_j^{\prime}-u_j\right), \quad \forall u_j^{\prime} \in \cA_j.
\end{equation*}  
 Similarly, we say 
 $f$ has a second-order linear derivative in $\cA_i\times \cA_j$ 
if $f$ has a linear derivative $ \frac{\delta f}{\delta u_i }$ in $\cA_i$, and 
 there exists 
  a function  
$\frac{\delta^2 f}{\delta u_i \delta u_j}:\cH^2(\sR^{kN})\times  \cH^2(\sR^k) \times  \cH^2(\sR^k) \to \sR$ such that 
 for all $\bu \in \cA$, 
$\frac{\delta^2 f}{\delta u_i \delta u_j}(\bu; \cdot, \cdot)$ is bilinear  and for all $u'_i\in \cH^2(\sR^k)$,
$\frac{\delta^2 f}{\delta u_i \delta u_j}(\bu; u'_i, \cdot)$
is the linear 
derivative of 
$\bu\mapsto \frac{\delta f}{\delta u_i }(\bu; u'_i)$ in $\cA_j$.

Using the notion of linear derivatives, the following theorem  constructs an $\alpha$-potential function
for the game $\cG$ and {\color{blue}quantifies} the associated $\alpha$.

\begin{prop}
\label{prop:general_potential}
Suppose that for all $i, j \in [N]$, $J_i$ has a   linear derivative $\frac{\delta J_i}{\delta u_i}$   in $\cA_i$, and 
a second-order linear derivative $\frac{\delta^2 J_i}{\delta u_i\delta u_j}$ in $\cA_i\times \cA_j$. Assume further that 
for all $u^{\prime}_i\in \cA_i$ and $u^{\prime \prime}_j\in \cA_j$,
$\cA\ni \bu\mapsto \frac{\delta^2 J_i}{\delta u_i\delta u_j}(\bu; u^{\prime}_i, u^{\prime \prime}_j)\in \sR$ is continuous. 
  Define
    $\Phi:\cA  \to \sR$ by 
    \begin{equation}
    \label{eq:potential_general}
        \Phi(\bu)=\int_0^1 \sum_{j=1}^N \frac{\delta J_j}{\delta u_j}\left(r \bu ; u_j\right) \mathrm{d} r.
    \end{equation}
    Then $\Phi$ is an $\alpha$-potential function of the game $\mathcal{G}$   with 
    \begin{equation}
    \label{eq:alpha_general}
         \alpha \le  \sup_{i \in [N],  u'_i\in \cA_i, {\color{blue} u_j''}\in \cA_j, \bu   \in \cA}
 \sum_{j=1}^N  \left|    \frac{\delta^2 J_i}{\delta u_i \delta u_j} \left(\bu ; u_i^{\prime} , u''_j  \right) 
-\frac{\delta^2 J_j}{\delta u_j \delta u_i}\left(\bu ; u''_j , u_i^{\prime} \right)   
 \right|.
    \end{equation}
\end{prop}
{\color{blue}\begin{proof}
Proposition \ref{prop:general_potential} follows as a special case of \cite[Theorem 2.5]{guo2024alpha}, using the specific choice $\bz = 0$.
With  this choice, the bound in \eqref{eq:alpha_general} is tighter than the general upper bound on $\alpha$ provided in \cite[Equation 1.3]{guo2024alpha}, as it no longer involves  a multiplicative constant  $2$ used in \cite{guo2024alpha}.
\end{proof} }
\section{$\alpha$-Potential Function
for Distributed Games}\label{sec:bound_alpha_open}

This section presents more explicit expressions of the $\alpha$-potential function \eqref{eq:potential_general} and the corresponding $\alpha$
from Proposition \ref{prop:general_potential}, expressed in terms of the model coefficients.

The following lemma analytically characterizes the linear derivatives of all players' objective functions. 
An important  tool is the derivative of each player’s controlled state with respect to her own control, defined by \eqref{eq:Y_state}.

\begin{lemma}
\label{lemma:linear_derivative}
    Suppose   (H.\ref{assum:regularity_general}) holds. For all $i \in [N]$,
   $J_i$ has a linear derivative 
$\frac{\delta J_i}{\delta u_i}:\cH^2(\sR^{kN})\times \cH^2(\sR^{k})\to \sR $ in $\cA_i$ satisfying for all $\bu \in \cA$ and $u'_i\in \cA_i$, 
\begin{equation}
 \label{eq:linear_derivative_general}
    \begin{aligned}
    \frac{\delta J_i}{\delta u_i}\left(\bu ; u_i^{\prime}\right)
      & = \mathbb{E}\Bigg[
    \int_0^T
    \begin{pmatrix}
        Y^{u'_i}_{i,t} \\
    u_{i,t}' 
    \end{pmatrix}^\top 
    \begin{pmatrix}
        \partial_{x_i} f_i \\
    \partial_{a_i} f_i
    \end{pmatrix}
     (t, \bX_t^{ \bu}, \bu_t  )\,
     \mathrm{d} t
     + (Y^{u'_i}_{i, T})^\top (\partial_{x_i} g_i)  (\bX_T^{ \bu})  \Bigg], 
    \end{aligned}  
\end{equation}
where $\bX^{\bu}\in \cS^2(\sR^{dN})$ satisfies 
\eqref{eq:X_state},
and
$Y_i^{u'_i} \in \cS^2(\sR^{d})$  satisfies the dynamics 
\begin{equation}\label{eq:Y_state}
\d Y_{i,t}   = b_i(t) u_{i,t}' \d t,
\quad t\in (0,T];
\quad Y_{i,0}  = 0.
\end{equation}

Moreover, 
for all $i, j \in [N]$ with $i\not = j$,
   $J_i$ has a second-order linear derivative 
$\frac{\delta^2 J_i}{\delta u_i\delta u_j}:\cH^2(\sR^{kN})\times \cH^2(\sR^{k})\times \cH^2(\sR^{k}) \to \sR $ in $\cA_i\times \cA_j$ satisfying for all $\bu \in \cA$, $u'_i\in \cA_i$
and $u''_j\in \cA_j$, 
\begin{equation}
\label{eq:2nd_derivative_J1}
    \begin{aligned}
 \frac{\delta^2 J_i}{\delta u_i\delta u_j}(\bu;u'_i,u''_j)
& =
\sE\left[ \int_0^T 
   \begin{pmatrix}
  Y^{u_i'}_{i,t}
  \\
   u_{i, t}'
\end{pmatrix}^\top
\begin{pmatrix}
 \partial^2_{x_i x_j} f_{i }      & \partial^2_{x_i a_j}  f_{i } 
\\
\partial^2_{a_i x_j}  f_{i  }    & \partial^2_{a_i a_j}  f_{i } 
\end{pmatrix}(t,\bX_t^{ \bu}, \bu_t)
 \begin{pmatrix}
  Y^{ u_j''}_{j,t}
  \\
  u''_{j, t}
\end{pmatrix}
 \d t\right]
 \\&\quad
+
\sE\left[ 
 ( Y^{u_i'}_{i,T})^\top  (\partial^2_{x_i x_j}  g_{i  }  ) (\bX^{\bu}_T)Y^{u_j''}_{j, T}
 \right].
    \end{aligned}
\end{equation}
\end{lemma}
{\color{blue} \begin{proof}
The proof follows directly from 
\cite[Lemma 4.8]{carmona2016lectures}, the convexity of $\cA_i$, and the linearity of the state dynamics  \eqref{eq:X_state}. 
\end{proof}

The expression \eqref{eq:2nd_derivative_J1} of the second-order derivative $ \frac{\delta^2 J_i}{\delta u_i\delta u_j}$ is simpler than the formula given in \cite[Equation 4.6]{guo2024alpha} for general stochastic differential games,
due to the fact that player $j$’s control does not affect player $i$’s state evolution.  
 }
 
Leveraging Lemma \ref{lemma:linear_derivative},  
the $\alpha$-potential function given in \eqref{eq:potential_general}  can be expressed as  
\begin{equation}
\label{eq:potential_ru}
  \begin{aligned}
 \Phi(\bu) &= \int_0^1 \sum_{i=1}^N   \frac{\delta J_i}{\delta u_i}\left(r\bu ; u_i \right) \d r
 \\
      & =\int_0^1 \sum_{i=1}^N    \mathbb{E}\Bigg[
    \int_0^T
    \begin{pmatrix}
        Y^{u_i}_{i,t} \\
    u_{i,t} 
    \end{pmatrix}^\top 
    \begin{pmatrix}
        \partial_{x_i} f_i \\
    \partial_{a_i} f_i
    \end{pmatrix}
     (t, \bX_t^{r \bu}, r \bu_t  )\,
     \mathrm{d} t
     + (Y^{u_i}_{i, T})^\top (\partial_{x_i} g_i)  (\bX_T^{ r \bu}) \Bigg] \d r, 
    \end{aligned}  
\end{equation}
which depends on  the aggregated behavior of   the  processes $(\bX^{r \bu})_{r\in [0,1]}$ parameterized 
by $r$.

To simplify the expression \eqref{eq:potential_ru}, 
the following lemma establishes a \emph{separation principle}, which  exploits the linearity of the   dynamics \eqref{eq:X_state} and \eqref{eq:Y_state}, and decomposes $\bX^{r \bu}$ into $\bX^{\bu}$
and $\bY^{\bu}$.

\begin{lemma}
\label{lemma:state_decomposition}
Suppose   (H.\ref{assum:regularity_general}) holds.
For all $\bu \in \cH^2(\sR^{kN})$ and $r\in [0,1]$,
$\bX^{r \bu} =  \bX^{\bu} -(1-r)  \bY^{\bu}$.
\end{lemma}
{\color{blue}\begin{proof}
The proof simply follows by noting that  the process
$\tilde{\bX} \coloneqq \bX^{\bu} -(1-r)  \bY^{\bu}$   has the same initial condition and satisfies the same dynamics as ${\bX}^{r\bu}$.   
\end{proof}}

Based on Lemma  \ref{lemma:state_decomposition}, 
the following theorem simplifies the expression \eqref{eq:potential_ru} of the $\alpha$-potential function, and derives an explicit   upper bound  for $\alpha$ in terms of the model coefficients.

\begin{theorem}
\label{thm:potential_diff}

Suppose (H.\ref{assum:regularity_general}) holds.
The  function $\Phi:\cA \to \sR$  in  \eqref{eq:potential_general}
can be expressed as
\begin{equation}\label{eqn:alpha_potential_function}
    \begin{aligned}
\Phi(\boldsymbol{u})
   =    \E  &\left[\int_0^T 
   F(t,\bX^{\bu}_t,\bY^{\bu}_t , \bu_t) \d  t + G(\bX^{\bu}_T, \bY^{\bu}_T)\right],
    \end{aligned}
\end{equation}
where for each $\bu =(u_i)_{i\in [N]}\in \cA$, 
$\bX^{\bu}=(X^{u_i}_i)_{i\in [N]}$ 
and 
$ \bY^{\bu}=(Y^{u_i}_i)_{i\in [N]}$
satisfy 
\eqref{eq:X_state} and 
\eqref{eq:Y_state}, respectively, and 
$F:[0,T]\times \sR^{dN}
\times \sR^{dN}
\times \sR^{kN}\to \sR$
and 
$G: \sR^{dN}
\times \sR^{dN}
\to \sR$ satisfy
 for all $t\in [0,T]$,  
 $x=(x_i)_{i\in [N]}, y=(y_i)_{i\in [N]}\in \sR^{dN}$ and $a=(a_i)_{i\in [N]}\in \sR^{kN}$,
\begin{align}
\label{eq:F_G_alpha_PG}
    \begin{split}
       F(t,x,y,a)&\coloneqq 
     \sum_{i=1}^N \int_0^1 
     \binom{ y_i}{a_i}^{\top}\binom{\partial_{x_i} f_i}{\partial_{a_i} f_i}\left(t, x - (1-r) y, r a\right) 
     \d r,
     \\
     G(x,y) & \coloneqq
     \sum_{i=1}^N \int_0^1
     y_i^{\top} \left(\partial_{x_i} g_i\right)\left(  x - (1-r) y\right)
       \d r.  
    \end{split}
\end{align}
Moreover,  
$\Phi$ is an $\alpha$-potential function of the game $\mathcal{G}$   with 
    \begin{equation}
    \label{eq:alpha_upper_bound_distributed}
        \begin{aligned}
            \alpha   & \leq \sup_{i\in[N]} \sum_{j\in [N]\setminus\{i\}} U_iU_j \Bigg(
T B_iB_j\|   \partial^2_{x_ix_j } \Delta^f_{i,j }\|_{L^\infty} +   
T^{\frac{1}{2}}B_i
\|   \partial^2_{x_ia_j } \Delta^f_{i,j }\|_{L^\infty} + 
T^{\frac{1}{2}}B_j
\|   \partial^2_{a_i x_j } \Delta^f_{i,j }\|_{L^\infty}  \\
 &\qquad+ \|   \partial^2_{a_i a_j } \Delta^f_{i,j }\|_{L^\infty} +
 B_iB_j \|   \partial^2_{x_ix_j } \Delta^g_{i,j }\|_{L^\infty}\Bigg),
        \end{aligned}
    \end{equation}
  where for all $i,j\in [N]$ with $i\not =j$, 
$
\Delta^f_{i,j} \coloneqq f_i - f_j$, $\Delta^g_{i,j} \coloneqq g_i - g_j$,   
  $B_i\coloneqq \|b_i\|_{L^2}$
  and 
  $U_i\coloneqq \sup_{u_i\in \cA_i}\|u_i\|_{\cH^2}$.
\end{theorem}

  \begin{proof}
 The expression 
\eqref{eqn:alpha_potential_function} follows by substituting the expression $\bX^{r \bu} =  \bX^{\bu} -(1-r)  \bY^{\bu}$ {\color{blue} from 
Lemma  \ref{lemma:state_decomposition} into \eqref{eq:potential_ru}},  
and applying   Fubini's theorem.

To get an upper bound of  $\alpha$, by Lemma \ref{lemma:linear_derivative},
   \begin{align}
\label{eq:V_jacobian_difference_open_loop}
\begin{split}
& \frac{\delta^2 J_i}{\delta u_i\delta u_j}(\bu;u'_i,u''_j)
-\frac{\delta^2 J_j}{\delta u_j\delta u_i}(\bu;u''_j,u'_i)
\\
&
=
\sE\left[ \int_0^T 
  \begin{pmatrix}
  Y^{u_i'}_{i,t}
  \\
   u_{i, t}'
\end{pmatrix}^\top
\begin{pmatrix}
 \partial^2_{x_i x_j} \Delta^f_{i,j }      & \partial^2_{x_i a_j} \Delta^f_{i,j } 
\\
\partial^2_{a_i x_j} \Delta^f_{i,j }    & \partial^2_{a_i a_j} \Delta^f_{i,j } 
\end{pmatrix}(t,\cdot)
 \begin{pmatrix}
  Y^{ u_j''}_{j,t}
  \\
  u''_{j, t}
\end{pmatrix}
 \d t +
 ( Y^{u_i'}_{i,T})^\top  (\partial^2_{x_i x_j} \Delta^g_{i,j }  ) (\bX^{\bu}_T)Y^{u_j''}_{j, T}
 \right],
 \end{split}
\end{align}
where we write   $ \partial^2_{x_ix_j} \Delta^f_{i,j }  (t, \cdot) =  \partial^2_{x_ix_j} (f_i-f_j) (t,\bX_t^\bu, \bu_t)
$ and similarly for other derivatives.    Moreover,  by \eqref{eq:Y_state}, for any $t\in [0,T]$, $Y_{i,t}^{u_i'} = \int_0^t b_i(v) u_i'(v) \d v$, and hence by the Cauchy-Schwarz inequality,
    \begin{equation}
        \begin{aligned}
            \E \left[  \left| Y_{i,t}^{u_i'} \right|^2 \right] & = \E \left[ \left| \int_0^t b_i(v) u_i'(v)  \d v \right|^2 \right] 
            \le 
             \E \left[   \int_0^t |b_i(v)|^2 \d v \int_0^t |u_i'(v)|^2  \d v  \right] 
           =  \| b_i\|_{L^2}^2 \|u_i' \|_{\cH^2}^2.
        \end{aligned}
    \end{equation}
Thus 
$\| Y_{i}^{u_i'}\|^2_{\cH^2}
\le 
 T\sup_{t\in [0,T]}\E \left[  \left| Y_{i,t}^{u_i'} \right|^2 \right]
 \le 
T \| b_i\|_{L^2}^2 \|u_i' \|_{\cH^2}^2$.

We now estimate each term in \eqref{eq:V_jacobian_difference_open_loop}.
 Observe that for all $t\in [0,T]$, 
\begin{align}
\begin{split}
   & \left| \sE\left[\int_0^T  \left(Y^{u'_i}_{i, t}\right)^\top  (  \partial^2_{x_i x_j } \Delta^f_{i,j } ) (t, \cdot)  Y^{u''_j}_{j, t} \d t\right]\right|
\le 
 \sE\left[\int_0^T  |Y^{u'_i}_{i, t}| \|(  \partial^2_{x_i x_j } \Delta^f_{i,j } ) (t, \cdot)\|_{\textrm{sp}}  |Y^{u''_j}_{j, t}| \d t\right] 
\\
&
\le 
\|   \partial^2_{x_ix_j } \Delta^f_{i,j }  \|_{L^\infty} 
 \sE\left[\int_0^T  |Y^{u'_i}_{i, t}| |Y^{u''_j}_{j, t}| \d t\right] 
\le \|   \partial^2_{x_ix_j } \Delta^f_{i,j }  \|_{L^\infty}  
\|Y^{u'_i}_{i}\|_{\cH^2} 
\| Y^{u''_j}_{j}\|_{\cH^2 }  
\\
&\leq  T \|   \partial^2_{x_ix_j } \Delta^f_{i,j }\|_{L^\infty}  \| b_i\|_{L^2}  \| b_j\|_{L^2}   \|u_i' \|_{\cH^2}  \|u_j'' \|_{\cH^2} .
\end{split}
\label{eq:cost_f_xx_open_loop}
\end{align}
Similarly, we have 
\begin{align}
\begin{split}
   & \Bigg| \sE\Bigg[\int_0^T  \left(Y^{u'_i}_{i, t}\right)^\top  (  \partial^2_{x_i a_j } \Delta^f_{i,j } ) (t, \cdot)  u_{j,t}'' \d t\Bigg]\Bigg|  \le \|   \partial^2_{x_i a_j } \Delta^f_{i,j }  \|_{L^\infty}
\|Y^{u'_i}_{i} \|_{\cH^2} 
\| u_{j}''\|_{\cH^2 }  
\\
&
\leq  T^\frac{1}{2} \|   \partial^2_{x_ia_j } \Delta^f_{i,j }\|_{L^\infty}  \| b_i\|_{L^2}     \|u_i' \|_{\cH^2}  \|u_j'' \|_{\cH^2},
\end{split}
\label{eq:cost_f_xx_open_loop2}
\end{align}
and that 
\begin{align}
\label{eq:cost_f_xx_open_loop3}
\begin{split}
   \left| \sE\left[\int_0^T  \left(u'_{i, t}\right)^\top  (  \partial^2_{a_i x_j } \Delta^f_{i,j } ) (t, \cdot)  Y_{j,t}^{u''_j} \d t\right]\right|
 & \leq  T^\frac{1}{2} \|   \partial^2_{a_i x_j } \Delta^f_{i,j }\|_{L^\infty}  \| b_j\|_{L^2}     \|u_i' \|_{\cH^2}  \|u_j'' \|_{\cH^2},
\\
\left| \sE\left[\int_0^T  \left(u'_{i, t}\right)^\top  (  \partial^2_{a_i a_j } \Delta^f_{i,j } ) (t, \cdot)  u_{j,t}'' \d t\right]\right|
&\leq  \|   \partial^2_{a_i a_j } \Delta^f_{i,j }\|_{L^\infty}      \|u_i' \|_{\cH^2}  \|u_j'' \|_{\cH^2}.
\end{split}
\end{align}
Finally, we have
\begin{equation}\label{eq:cost_g_xx_open_loop}
\begin{aligned}
    \sE\left[ 
 ( Y^{u_i'}_{i,T})^\top  (\partial^2_{x_i x_j} \Delta^g_{i,j }  ) (\bX^{\bu}_T)Y^{u_j''}_{j, T} \right] 
 &\leq 
 \| \partial^2_{x_i x_j} \Delta^g_{i,j }\|_{L^\infty}  \sE\left[ 
 | Y^{u_i'}_{i,T}|  |Y^{u_j''}_{j, T} |\right] 
 \\
 &\le 
 \| \partial^2_{x_i x_j} \Delta^g_{i,j }\|_{L^\infty}    \| b_i\|_{L^2}  \| b_j\|_{L^2}   \|u_i' \|_{\cH^2}  \|u_j'' \|_{\cH^2}.
 \end{aligned}
\end{equation}
Combining  \eqref{eq:cost_f_xx_open_loop}, \eqref{eq:cost_f_xx_open_loop2}, \eqref{eq:cost_f_xx_open_loop3}, \eqref{eq:cost_g_xx_open_loop}, and  Proposition \ref{prop:general_potential} yields the desired result. 
\end{proof}

 Compared with \eqref{eq:potential_ru}, \eqref{eqn:alpha_potential_function} isolates the contribution of $r$ and expresses the $\alpha$-potential function only in terms of  $\bX^\bu$ and $\bY^\bu$.
 This reformulation enables the use of standard control techniques to minimize the 
$\alpha$-potential function, thereby simplifying the computation of approximate Nash equilibria for the game   $\cG$. 

To see it, recall that 
the objective function \eqref{eq:potential_ru} depends on the aggregated behavior of the state processes $\bX^{r\bu}$
with respect to $r\in [0,1]$.
This parameter $r$
 acts as  a  uniformly distributed noise independent of $\sF$. 
As shown in \cite{guo2024alpha}, to find a minimizer of \eqref{eq:potential_ru} that is adapted to $\sF$, one must lift the problem into a conditional McKean–Vlasov control framework, where the state variable becomes the conditional law 
$\mathcal L(\bX^{r\bu}, \bY^{\bu}, r\mid \sF)$. The resulting optimal control  is characterized by an infinite-dimensional Hamilton–Jacobi–Bellman (HJB) equation defined on the space of probability measures.

In contrast, the reformulated objective \eqref{eqn:alpha_potential_function} depends only on the $2dN$-dimensional state variables $(\bX^{\bu}, \bY^{\bu})$.  The corresponding optimal control can then be characterized by a standard HJB equation defined on the space   $  \sR^{2dN}$, as    will be shown in Section \ref{sec:optimize_potential_function}.

We further remark that when the upper bound in \eqref{eq:alpha_upper_bound_distributed} is zero, the game $\cG$ becomes a potential game, and its Nash equilibria can be obtained by minimizing a potential function that  involves \emph{only on the state variable $\bX^{\bu}$}, as defined in \eqref{eq:potential_symmetry}.

\begin{theorem}
    Suppose  (H.\ref{assum:regularity_general}) holds, and for all $i,j\in [N]$ with $i\not = j $,
   \begin{equation}
   \label{eq:symmetry_condition}
    \partial^2_{x_ix_j } f_i = \partial^2_{x_ix_j } f_j,
   \quad 
   \partial^2_{a_i x_j } f_i = \partial^2_{a_i x_j }  f_j,
   \quad 
   \partial^2_{a_i a_j } f_i =  \partial^2_{a_i a_j }  f_j,
   \quad 
   \partial^2_{x_ix_j }  g_i= \partial^2_{x_ix_j }   g_j. 
   \end{equation}
Then the game $\cG$ is a potential game with a potential function defined by
\begin{equation}
\label{eq:potential_symmetry}
    \bar {\Phi}(\bu) \coloneqq 
        \E  \left[\int_0^T 
   \bar F(t,\bX^{\bu}_t,\bu_t) \d  t + \bar G(\bX^{\bu}_T)\right],
\end{equation}
where 
$\bX^{\bu}$ 
satisfies 
\eqref{eq:X_state}, $\bar F:[0,T]\times \sR^{dN} 
\times \sR^{kN}\to \sR$
and 
$\bar G:   \sR^{dN}
\to \sR$ satisfy
 for all $t\in [0,T]$,  
 $x=(x_i)_{i\in [N]} \in \sR^{dN}$ and $a=(a_i)_{i\in [N]}\in \sR^{kN}$,
\begin{align*}
    \begin{split}
       \bar F(t,x, a)&\coloneqq 
     \sum_{i=1}^N \int_0^1 
     \binom{ x_i}{a_i}^{\top}\binom{\partial_{x_i} f_i}{\partial_{a_i} f_i}\left(t, r x, r a\right) 
     \d r, 
     \quad  
     G(x) \coloneqq
     \sum_{i=1}^N \int_0^1
     x_i^{\top} \left(\partial_{x_i} g_i\right)\left(  r x \right)
       \d r.  
    \end{split}
\end{align*}
Moreover,  for all $\bu \in \cA$,
\begin{equation}
\label{eq:difference_potential}
 \bar \Phi(\bu) = \Phi(\bu )+
     \E  \left[\int_0^T 
   \bar F(t,\bX^{\boldsymbol{0}}_t,\boldsymbol{0} ) \d  t + \bar G(\bX^{\boldsymbol{0} }_T)\right],
\end{equation}
where   $\Phi$    is defined in \eqref{eqn:alpha_potential_function}
and $\boldsymbol{0}$ is   the constant process taking the value zero at all times.
\end{theorem}

\begin{proof}
Under the symmetry condition \eqref{eq:symmetry_condition},
the fact that the function $\bar \Phi$  in \eqref{eq:potential_symmetry} is a potential function for the game $\cG$
follows from   analogous    arguments   to those used for distributed games with Markov policies in \cite[Theorem 3.2]{guo2025towards}. 
Since both $\bar \Phi$ and $\Phi$ are potential functions for the game $\cG$, it holds for all $u\in \cA$,
$$
\bar \Phi(\bu)- \Phi(\bu)= \bar \Phi(\boldsymbol{0})- \Phi(\boldsymbol{0}).
$$
To see it, 
assume without loss of generality that $N=2$. Then for all $\bu=(u_i)_{i=1}^2\in \cA$, using the definition \eqref{eq:alpha_v_phi} of a potential function,  
\begin{align*}
    \bar \Phi(\bu)- \Phi(\bu)
    &=\bar \Phi((u_1,u_2))-
    \bar \Phi((0,u_2))+\bar \Phi((0,u_2))
    -\bar \Phi(\boldsymbol{0})+\bar \Phi(\boldsymbol{0})
    \\
  &\quad   -\left(
       \Phi((u_1,u_2))-
     \Phi((0,u_2))+ \Phi((0,u_2))
    - \Phi(\boldsymbol{0})+ \Phi(\boldsymbol{0})
    \right)
    \\
    &= J_1((u_1,u_2))-
    J_1((0,u_2))+J_2((0,u_2))
    -J_2((0,0))+\bar \Phi(\boldsymbol{0})
    \\
  &\quad   -\left(
   J_1((u_1,u_2))-
    J_1((0,u_2))+J_2((0,u_2))
    -J_2((0,0))
    -J_2(0,0)+ \Phi(\boldsymbol{0})
    \right)
    \\
    &=\bar \Phi(\boldsymbol{0})- \Phi(\boldsymbol{0}).
\end{align*}
The desired identity \eqref{eq:difference_potential} then follows from the fact that   $\bY^{\boldsymbol{0}}_t=0$ for all $t\in [0,T]$,
and $F(t,x,0,0)=G(x,0)$ for all $(t,x)\in [0,T]\times \sR^{dN}$. 
\end{proof}

\section{Optimize $\alpha$-Potential Function for $\alpha$-NE}
\label{sec:optimize_potential_function}

Given the  $\alpha$-potential function $\Phi$ defined in \eqref{eqn:alpha_potential_function},
this   section characterizes its minimizer   over the admissible control space $\cA$, which in turn constructs  analytically  an $\alpha$-NE for the distributed game $\mathcal{G}$.
We adopt a dynamic programming approach that characterizes the minimizer of the $\alpha$-potential function in feedback form via solutions to suitable  HJB
integro-partial differential equations. This characterization offers a theoretical foundation for  developing  policy gradient algorithms 
to solve the distributed game $\mathcal{G}$; see Section  \ref{sec:PG} for details.

 More precisely, we consider 
 the following control problem
\begin{equation} \label{eq:min_Phi_F_G}
    \begin{aligned}
\inf_{\bu \in \cA} \Phi(\bu),
\quad \Phi(\bu) =\E  &\left[\int_0^T 
   F(t,\bX^{\bu}_t,\bY^{\bu}_t , \bu_t) \d  t + G(\bX^{\bu}_T, \bY^{\bu}_T)\right],
    \end{aligned}
\end{equation}
where    
$\cA$ is the set of admissible controls given by
\begin{equation*}
    \cA\coloneqq 
\{u:\Omega\times [0,T]\to A 
\mid   
 u\in \cH^2(\sR^{kN})
\},
\end{equation*}
$F$ and 
$G $ are defined in \eqref{eq:F_G_alpha_PG}, and $(\bX^{\bu}, \bY^{\bu})$ satisfy the following  state dynamics: 
\begin{equation} \label{eq:vec_X_Y}
  \left\{
   \begin{aligned}
 \d  \bX_t &= b(t) \bu_t \d t + \sigma(t) \d W_t 
+ 
      \sum_{j=1}^m\int_{\sR_0^p}
      {\gamma}_{j}(t, z) \widetilde {\eta}_{j}(\d t, \d z), 
\quad 
\bX_0=\boldsymbol{\xi},
 \\ 
\d \bY_t &= b(t) \bu_t \d t, 
\quad \bY_0=0,
\end{aligned}
\right.
\end{equation}
where $\boldsymbol{\xi}=(\xi^\top_1,\ldots, \xi^\top_N)^\top$, 
and for all 
$t\in [0,T]$ and $z\in \sR^p_0$,
$$
b(t) \coloneqq \diag(b_1(t),\cdots,b_N(t))\in \sR^{dN\times kN}, \quad 
\sigma (t) \coloneqq \begin{pmatrix}
    \sigma_1(t) \\ \vdots\\\sigma_N(t)
\end{pmatrix}\in \sR^{dN\times n},
\quad 
\gamma_j(t,z) \coloneqq \begin{pmatrix}
    \gamma_{1j}(t,z) \\ \vdots \\ \gamma_{Nj}(t,z)
\end{pmatrix}\in \sR^{dN}.
$$

\subsection{Verification theorem}
The minimizer of \eqref{eq:min_Phi_F_G}
can be constructed by standard verification results. 
To see it, let 
$\mathcal C^{1,2,1}([0,T]\times \sR^{dN}\times \sR^{dN})$
be the space of functions 
$\phi=\phi(t,x,y):[0,T]\times 
\sR^{dN}\times \sR^{dN} \to \sR$ such that 
$\partial_t \phi$, $\partial_x \phi$, 
$\partial^2_{xx} \phi$, and 
$\partial_y \phi$ exist and are continuous and there exists    $C\ge 0$
 such that for all $(t,x, y)\in [0,T]\times \sR^{dN}\times \sR^{dN}$,
$ |\phi(t,x,y)|+|\partial_t \phi(t,x,y)|\le C(1+|x|^2+|y|^2)$,
 $ |\partial_x \phi(t,x,y)|+|\partial_y \phi(t,x,y)|\le C(|1+|x|+|y|)$,
 and 
  $ |\partial^2_x \phi(t,x,y)| \le C$. 
For all $a\in A$
and $\phi \in \mathcal C^{1,2,1}([0,T]\times \sR^{dN}\times \sR^{dN})$,
define the operator 
$\mathbb L^a \phi: [0,T] \times \sR^{dN} \times \sR^{dN} \ra \sR $ by
\begin{equation*}\label{eq:Lu}
\begin{aligned}
     \mathbb L^a \phi(t,x,y) \coloneqq & \left(b(t) a\right)^\top \left(\partial_x  \phi(t,x,y) + \partial_y \phi(t,x,y)\right) + \frac{1}{2} \Tr \left( \sigma(t) \sigma(t)^\top \partial_{xx}^2 \phi(t,x,y) \right) \\
        &+ \sum_{j=1}^m \int_{\sR^p_0}  \left( \phi(t, x + \gamma_j (t, z), y) -\phi(t,x,y) - \partial_x \phi (t,x,y)^\top \gamma_j (t,z)\right) \nu_j (\d z ),
\end{aligned}
\end{equation*}
and define the associated Hamiltonian by 
\begin{equation*}\label{eq:Hamiltonian}
     H(t, x, y,\phi, a) = \mathbb L^a \phi(t,x,y) + F(t,x,y,a).
\end{equation*}
The   HJB equation  associated with \eqref{eq:min_Phi_F_G}
       is given by
       \begin{equation}
       \label{eq:HJB_full_info}
    \left\{
    \begin{aligned}
     &\partial_t w (t,x,y) + \inf_{a \in A} H(t, x, y,   w,a) = 0, \quad (t,x,y)\in [0,T]\times \sR^{dN}\times \sR^{dN}, \\
     &w(T,x,y) = G(x,y), \quad  (x,y)\in \sR^{dN}\times \sR^{dN},
    \end{aligned}
    \right.
\end{equation}

We now present the verification theorem, which  constructs an optimal control of \eqref{eq:min_Phi_F_G}
(and hence an $\alpha$-NE of the game $\cG$) 
in a 
feedback form using a smooth solution to the  HJB equation \eqref{eq:HJB_full_info}.

\begin{theorem}
\label{thm:verification_full}
    Suppose (H.\ref{assum:regularity_general}) holds. Assume that there exists $v\in C^{1,2,1}([0,T]\times \sR^{dN}\times \sR^{dN}) $
    such that 
    $\inf_{a\in A} H (t, x, y,  v, a)  \in \sR$ for all $(t,x,y)\in [0,T]\times \sR^{dN}\times \sR^{dN}$,
    and $v$ satisfies 
    the  HJB equation  \eqref{eq:HJB_full_info}.
Assume further that 
there exists a measurable map $\hat{a} :[0, T ]\times  \sR^{dN} \times \sR^{dN}\to  A$ such that 
\begin{equation}\label{eq:hat_a_full}
    \hat{a}(t,x,y) = \arg\min_{a\in A} H (t,x,y,  v,a), 
    \quad (t,x,y)\in [0,T]\times \sR^{dN}\times \sR^{dN},
\end{equation}
the corresponding controlled dynamics    
\begin{equation} 
\label{eq:state_feedback_full}
  \left\{
   \begin{aligned}
 \d  \bX_t &= b(t) \hat{a}(t,\bX_t,\bY_t) \d t + \sigma(t) \d W_t 
+ 
      \sum_{j=1}^m\int_{\sR_0^p}
      {\gamma}_{j}(t, z) \widetilde {\eta}_{j}(\d t, \d z), 
\quad 
\bX_0=\boldsymbol{\xi},
 \\ 
\d \bY_t &= b(t) \hat{a}(t,\bX_t,\bY_t)  \d t, 
\quad \bY_0=0,
\end{aligned}
\right.
\end{equation}
has 
a square integrable strong solution 
$(\hat{\bX}, \hat{\bY})$
and that 
     the   control 
 $\hat\bu_t\coloneqq  \hat a (t, \hat{\bX}_t, \hat{\bY}_t)  $,
 $t\in [0,T]$,
 is in $\cH^2(\sR^{kN})$. 
 Then 
 $v(0,\xi, {\color{blue} 0})= \inf_{\bu \in \cA} \Phi(\bu)$,
 and
 $\hat \bu $ is an optimal control of \eqref{eq:min_Phi_F_G} 
 and an $\alpha$-NE of the distributed game $\cG$, with $\alpha$ given in \eqref{eq:alpha_upper_bound_distributed}.
\end{theorem}

 Theorem \ref{thm:verification_full} indicates that under sufficient regularity conditions,
an $\alpha$-NE for the game $\cG$   can be obtained by minimizing the $\alpha$-potential function $\Phi$ in \eqref{eq:min_Phi_F_G}   over feedback controls of the form $\bu_t = \phi(t, \bX_t, \bY_t)$, where 
$\phi:[0,T]\times \sR^{dN}\times \sR^{dN}\to A$
is a sufficiently regular policy profile, and 
$(\bX, \bY)$ satisfies  \eqref{eq:state_feedback_full}
with $\hat a$ replaced by $\phi$.
This result provides the theoretical foundation for the policy gradient algorithm presented in Section \ref{sec:PG}.

The proof of Theorem \ref{thm:verification_full}
follows from   standard verification arguments for classical stochastic control problems (see, e.g., \cite[Chapter 5]{yong2012stochastic}). The first step is to show that $\Phi(\bu) \ge v(0, \xi, 0)$ for all $\bu \in \cA$, by applying It\^o's formula for jump-diffusion processes to the function $t \mapsto v(t, \bX^{\bu}_t, \bY^{\bu}_t)$ and using the HJB equation \eqref{eq:HJB_full_info} satisfied by $v$. The second step is to  show that $\Phi(\hat{\bu}) = v(0, \xi, 0)$ due to  the definition \eqref{eq:hat_a_full} of $\hat{a}$, 
   which implies the    optimality of $\hat \bu$.

\subsection{Viscosity characterization }
In the case where the HJB equation \eqref{eq:HJB_full_info}  does not admit a classical solution, we can characterize the value function of  
\eqref{eq:min_Phi_F_G} as the continuous viscosity solution of \eqref{eq:HJB_full_info}. 
To this end, define
the value function  starting from time $t\in [0,T]$ and state $(x,y)\in \sR^{dN}\times \sR^{dN}$  by
\begin{equation} \label{eq:min_Phi_F_G_t}
    \begin{aligned}
V(t,x,y)\coloneqq \inf_{\bu \in \cA} \E  &\left[\int_t^T 
   F(s,\bX^{\bu}_t,\bY^{\bu}_s , \bu_s) \d  s + G(\bX^{\bu}_T, \bY^{\bu}_T)\bigg\vert X^{\bu}_t =x,
   Y^{\bu}_t =y\right].
    \end{aligned}
\end{equation}
We    impose the following   assumptions, which are standard in the literature for establishing the uniqueness of   viscosity solutions (see e.g., \cite{pham1998optimal,dumitrescu2017mixed, jakobsen2005continuous}).

\begin{assumption}
\label{assum:full_viscosity}
  Assume the setting in (H.\ref{assum:regularity_general}).  For all $i,j\in [N]$, $A_i$ is compact, $b_i$  $\sigma_i$, and   $\gamma_{ij}$ are continuous in $t$,
    and $\partial_{x_i}f_i$ and 
$\partial_{a_i}f_i$
are continuous in  all variables. 

\end{assumption}

Now we  identify the value function $V$ defined by 
\eqref{eq:min_Phi_F_G_t}
as the unique viscosity solution to \eqref{eq:HJB_full_info}.
{\color{blue} We first recall the definition of viscosity solutions as given in \cite[Definition 2.1]{jakobsen2005continuous}:
\begin{definition} 
    A function $v: [0,T]\times \sR^{dN}\times \sR^{dN} \to \sR$ is called a  viscosity subsolution
    (resp.~supersolution)
    of \eqref{eq:HJB_full_info} if
    $v$ is   upper semicontinuous
    (resp.~lower semicontinuous)
    and 
    for every  
    $(t_0,x_0,y_0)\in [0,T)\times \sR^{dN}\times \sR^{dN}$
    and  
    $\phi \in C^{1,2,1}([0,T]\times \sR^{dN}\times \sR^{dN})$
    such that
    $\phi -v$  attains its minimum 
    (resp.~maximum)
    at $ (t_0,x_0,y_0)$, 
    $$
    \partial_t \phi(t_0,x_0,y_0)+\inf_{a\in A}H (t_0,x_0,y_0,\phi, a)\ge 0 \quad 
    \textnormal{(resp.~$\le 0)$}.
    $$
     {\color{blue}  If $v$ is both a viscosity subsolution and a viscosity supersolution of
\eqref{eq:HJB_full_info}, then $v$ is called a viscosity solution.}
\end{definition}}

\begin{theorem}
\label{thm:hjb_full_viscosity}
    Suppose (H.\ref{assum:full_viscosity}) holds. 
    The function $V$ defined by \eqref{eq:min_Phi_F_G_t} is the unique  viscosity solution of 
    the HJB equation  \eqref{eq:HJB_full_info}  
    in 
the class of  continuous functions
with at most quadratic 
growth in $(x,y)$,
in the sense that 
$V$ is a viscosity sub- and supersolution of 
\eqref{eq:HJB_full_info}    with terminal condition $V(T,x,y)=G(x,y)$.

\end{theorem}

{\color{blue}
\begin{proof}
The proof   
follows from standard arguments in viscosity solution theory, adapting results for control problems with   Lipschitz continuous costs (see, e.g., \cite{pham1998optimal, dumitrescu2017mixed}) to the present setting with locally Lipschitz continuous costs.
Indeed, 
under (H.\ref{assum:full_viscosity}), the functions $F$ and $G$
  defined in \eqref{eq:F_G_alpha_PG}    are continuous,
 and there exists   $C\ge 0$ such  that for all $t\in [0,T]$, $(x,y),(x',y')\in \sR^{dN}\times \sR^{dN}$ and $a\in A$,
\begin{align}
&|F(t,x,y,a)|+|G(x,y)|
 \le C(1+|x|^2+|y|^2),
 \label{eq:F_G_quadratic}\\
&|F(t,x,y,a)-F(t,x',y',a)|+|G(x,y)-G(x',y')|
\nonumber
\\
&\quad \le 
C(1+|x|+|y|+|x'|+|y'|)(|x-x'|+|y-y'|).
\label{eq:F_G_local_Lipschitz}
\end{align}
To see it, a direct computation 
and the boundedness of $\partial_{xx}g_i$ imply that 
\begin{align*}
|\partial_{x_i}G(x,y)|
&=
\left|\sum_{j=1}^N \int_0^1 (\partial_{x_ix_j}g_j)(x-(1-r)y)  y_j\d r
\right|\le C|y|,  
\\
|\partial_{y_i}G(x,y)|
&=
\left|
\int_0^1
(\partial_{x_i}g_i)(x-(1-r)y)\d r
- \sum_{j=1}^N \int_0^1 (1-r) (\partial_{x_ix_j}g_j)(x-(1-r)y)  y_j\d r
\right|
\\
&\le C(1+|x|+|y|),
\end{align*}
for a constant $C\ge 0$ independent of $(t,x,y)$.
This along with the mean value theorem proves the local Lipschitz continuity 
and the quadratic growth 
of $G$ given in 
\eqref{eq:F_G_quadratic} and 
\eqref{eq:F_G_local_Lipschitz}. Applying a similar argument to $F$ establishes the analogous regularity of $F$.

 By \eqref{eq:F_G_quadratic} and \cite[Lemmas 4.2]{dumitrescu2017mixed},   the value function   
$V$ has at most quadratic growth in $(x,y)$.  Define the upper   semicontinuous envelope  $V^*$ and the lower semicontinuous
envelope $V_*$ of $V$ such that for all $(t,x,y)\in [0,T]\times \sR^{dN}\times \sR^{dN}$,
$$
V^*(t,x,y)=\limsup_{(t',x',y')\to (t,x,y)}
V(t,x,y),
\quad 
V_*(t,x,y)=\liminf_{(t',x',y')\to (t,x,y)}
V(t,x,y).
$$
The quadratic growth of $V$ implies that $V^*$ and $V_*$ have   at most quadratic growth.
As shown in  \cite[Theorem 5.2]{dumitrescu2017mixed}, 
by the  weak dynamic programming principle,
$V^*$ and $V_*$ are   viscosity subsolution and supersolution of \eqref{eq:HJB_full_info}, respectively. 
The 
local Lipschitz continuity of $F$
and $G$ in \eqref{eq:F_G_local_Lipschitz} and the
strong comparison principle in 
\cite[Theorem 4.3]{jakobsen2005continuous}
  imply that $V^*\le V_*$, which along with the fact that 
$V^*\ge V\ge V_*$
yields that $V$  is the unique continuous viscosity solution. 
\end{proof}  
}

\section{Policy Gradient Algorithm for $\alpha$-NE}
\label{sec:PG}

Theorem
\ref{thm:verification_full}   characterizes an open-loop 
$\alpha$-NE for the distributed game $\cG$ in the    feedback form with respect to  the state process $\bX$ and the sensitivity process $\bY$. The feedback controls therein are constructed from solutions   to the corresponding HJB equations, which may not admit closed-form expressions.

In this section, we propose a policy gradient algorithm to compute the  $\alpha$-NE for the distributed game $\mathcal{G}$. The algorithm searches for the $\alpha$-NE by directly minimizing the $\alpha$-potential function \eqref{eqn:alpha_potential_function} over suitable parametric families. 
For clarity of exposition, we present the algorithm under the assumption that the jump measures $(\nu_j)_{j=1}^m$
in \eqref{eq:X_state}
are finite, i.e., 
$$
\nu_j(\sR^p_0)<\infty, 
\quad \forall j=1,\ldots, m.
$$
Problems involving singular jump measures with infinitely many jumps can be reduced to ones with finite-activity measures by applying the standard diffusion approximation (see, e.g., \cite{cont2005finite, dumitrescu2021approximation, reisinger2021penalty}). This approach involves truncating the singular measures at a given threshold and approximating the small-jump component using a modified diffusion coefficient. The approximation error depends on the choice of truncation threshold and the singularity of the jump measures $(\nu_i)_{i=1}^m$ near zero (see e.g., \cite[Lemma C.3]{dumitrescu2021approximation}).

{\color{blue} The algorithm begins by approximating the NE policy  given in Theorem
\ref{thm:verification_full}   
using a sufficiently expressive parametric family (e.g., a family of  deep neural networks) \cite{han2016deep, han2017deep}}. 
 Specifically, we consider  
 a family of policy profiles
 $\phi_\theta: [0,T] \times \mathbb{R}^{dN} \times \mathbb{R}^{dN} \to A$ with  
   weights $\theta \in  \mathbb{R}^L$,
 and  consider   
 for each $\theta\in \sR^L$, 
\begin{equation} \label{eq:min_Phi_F_G_theta}
    \begin{aligned}
 \Phi(\theta) \coloneqq \E  &\left[\int_0^T 
   F(t,\bX^{\theta}_t,\bY^{\theta}_t , \phi_\theta(t, \mathbf{X}^\theta_t, \mathbf{Y}^\theta_t)) \d  t + G(\bX^{\theta}_T, \bY^{\theta}_T)\right],
    \end{aligned}
\end{equation}
where $(\bX^\theta, \bY^\theta)$ 
are the state and sensitivity processes satisfying  the following dynamics:
\begin{equation} 
\label{eq:state_feedback_phi}
  \left\{
   \begin{aligned}
 \d  \bX_t &= b(t) \phi_\theta(t,\bX_t,\bY_t) \d t + \sigma(t) \d W_t 
+ 
      \sum_{j=1}^m\int_{\sR_0^p}
      {\gamma}_{j}(t, z) \widetilde {\eta}_{j}(\d t, \d z), 
\quad 
\bX_0=\boldsymbol{\xi},
 \\ 
\d \bY_t &= b(t) \phi_\theta(t,\bX_t,\bY_t)  \d t, 
\quad \bY_0=0,
\end{aligned}
\right.
\end{equation}
That is, we restrict   the control problem \eqref{eq:min_Phi_F_G} on the set of controls 
$\bu_t = \phi_\theta(t, \mathbf{X}^\theta_t, \mathbf{Y}^\theta_t)$, $t\in [0,T]$, induced by $\phi_\theta$.

  We     seek an  optimal policy that  minimizes \eqref{eq:min_Phi_F_G_theta}, 
 which  yields an approximate NE of the distributed game $\cG$
 as shown in Lemma \ref{lemma:a_potential_game} and  Theorem \ref{thm:verification_full}.
 This is achieved by 
 performing gradient descent
of \eqref{eq:min_Phi_F_G_theta} with respect to the weights $\theta$ based on simulated trajectories of \eqref{eq:state_feedback_phi}.  
 More precisely, given a fixed policy $\phi_\theta$, we consider the following Euler-Maruyama  approximation of  \eqref{eq:state_feedback_phi}  on the time grid $\pi_P\coloneqq \{0=t_0<\ldots<t_P=T\}$ for some $P\in \sN$:
for all $i\in [N]$,
let $X^{\theta}_{i,0}=\xi_i$
and $Y^{\theta}_{i,0}=0$, 
 and 
 for all $\ell=0,\ldots, P-1$,
\begin{equation}\label{eq:state_theta}
\begin{aligned}
        X^{\theta}_{i,t_{\ell+1}}
      &= 
       X^{\theta}_{i,t_{\ell}}
       +
      b_i(t_{\ell}) \phi_{\theta}(t_{\ell},\bX^{\theta}_{t_{\ell}},\bY^{\theta}_{t_{\ell}})  \Delta_\ell  + \sigma_i(t_{\ell}) \Delta W_{\ell} 
      \\
      &\quad 
      + 
      \sum_{j=1}^m
      \left(\sum_{k=N_{j,\ell} +1}^{ N_{j,\ell+1}}
      \gamma_{ij}(t_\ell, z_k)  
      -\Delta_{\ell} \int_{\sR^p_0}\gamma_{ij}(t_\ell, z)\nu(\d z)
      \right),
      \\
       Y^{\theta}_{i,t_{\ell+1}}
      &= 
       Y^{\theta}_{i,t_{\ell}}
       +
      b_i(t_{\ell}) \phi_{\theta}(t_{\ell},\bX^{\theta}_{t_{\ell}},\bY^{\theta}_{t_{\ell}})  \Delta_\ell ,\quad 
      \bX^{\theta}_{t_\ell}   =(X^{\theta}_{i,t_\ell})_{i\in [N]},
      \quad 
      \bY^{\theta}_{t_\ell}  =(Y^{\theta}_{i,t_\ell})_{i\in [N]},
\end{aligned}
\end{equation}
where $\Delta_\ell \coloneqq t_{\ell+1}-t_{\ell}$,
$\Delta W_\ell \coloneqq W_{t_{\ell+1}}-W_{t_{\ell}}$, 
$N_{j,\ell} $
denotes the  number of   jumps of the $j$-th Poisson random measure 
occurring 
over the  time interval 
$[0,t_\ell]$, and $z_k$
is the size of the $k$-th  jump sampled from the  distribution $\nu/\nu(\sR^p_0)$. 
Let $(\bX^{\theta,(  m)},\bY^{\theta,(m)})_{m=1}^M$, $M\in \sN$, 
be independent trajectories of \eqref{eq:state_theta}
with policy $\phi_\theta$,
and define the following empirical approximation of \eqref{eq:min_Phi_F_G_theta} 
\begin{equation}\label{eq:ComputationalCost}
     \Phi_M(\theta)\coloneqq \frac{1}{M}\sum_{m=1}^M\left[\sum_{\ell=0}^{P-1} F\left(\bX^{\theta, (m)}_{t_\ell}, \bY^{\theta, (m)}_{t_\ell},  \phi_\theta\left(t_\ell, \bX^{\theta, (m)}_{t_\ell}, \bY^{\theta, (m)}_{t_\ell} \right)\right) \Delta_{\ell}+ G\left(\bX^{\theta, (m)}_{t_P}, \bY^{\theta, (m)}_{t_P}\right)\right].
 \end{equation}
 By choosing a sufficiently large $M$
 and minimizing \eqref{eq:ComputationalCost}  
over $\theta$, we obtain 
an approximate minimizer of the $\alpha$-potential function, and consequently an approximate NE for the game $\cG$. 

Here we summarize the above policy gradient algorithm for the $\alpha$-potential game $\cG$. For simplicity, we  present a version of the algorithm that minimizes \eqref{eq:ComputationalCost} using a mini-batch stochastic gradient descent method. In practice, more sophisticated variants of stochastic gradient descent (such as Adam \cite{kingma2017adammethodstochasticoptimization}) can be employed to optimize \eqref{eq:ComputationalCost} more efficiently. 
 
\begin{algorithm}[H]
\caption{Policy Gradient Algorithm for $\alpha$-Potential Distributed Game $\cG$}
\label{alg:PG}
\begin{algorithmic}[1]
\State \textbf{Input:} 
A policy class $\{\phi_\theta :[0,T]\times \sR^{dN}\times \sR^{dN}\to A \mid 
\theta\in \mathbb{R}^L\}$,
  time grid $\pi_P$,
mini-batch sample size $M\in \sN$,
and learning rates $(\tau_n)_{n\ge 0}\subset (0,\infty)$.

\State \textbf{Initialize:} initial parameter    $\theta_0$.

\For{$n=0,1, \ldots$}
\State Generate 
$M$ independent trajectories from 
\eqref{eq:state_theta}
with policy  $\phi_{\theta_n}$. 
\State Evaluate the cost $J_M({\theta_n})$ by \eqref{eq:ComputationalCost} using the sampled trajectories.  
    \State Update $\theta$: 
    $ 
    \theta_{n+1}= \theta_n - \tau_n \nabla_\theta J_M(\theta_{n})
    $.
\EndFor
\State \textbf{Output:} approximate   policy $\phi_{\theta^*}$.
\end{algorithmic}
\end{algorithm}

Note that
at each iteration,
Algorithm \ref{alg:PG}   performs a  gradient descent update for all players’ policy parameters simultaneously.
In comparison, the standard fictitious play algorithm (see \cite{hu2021deep}) entails a significantly higher computational cost, as 
it requires solving $N$ individual stochastic control problems 
 at each iteration 
for each player’s best response to   other players' previous controls. 
 Each of these sub-problems typically requires hundreds or even thousands of gradient descent updates.

The $\alpha$-potential structure of the game $\cG$ is essential in reducing the computation of $\alpha$-NEs to the minimization of a   common objective function $\Phi$. This structure is key to ensuring the convergence of the gradient-based updates in Algorithm \ref{alg:PG}. While policy gradient methods  converge for various stochastic control problems (see e.g., \cite{reisinger2023linear, giegrich2024convergence, sethi2024entropy}), it is well known  that they may diverge in general multi-agent games without additional  structure assumptions \cite{mazumdar2020policy}.

\section{Application to Game-theoretic  Motion Planning}
\label{sec:crowd_motion}

This section illustrates our results  using the crowd motion game from Section \ref{sec:intro},
which is a
special case of the distributed games introduced in Section \ref{sec:setup}.
These games offer an agent-based framework for modeling crowd dynamics, where each pedestrian makes rational decisions to control their motion based on individual preferences, and the resulting equilibrium behavior determines the evolution of the crowd.

Specifically, 
given a joint control profile  $\bu =(u_i)_{i\in [N]}\in \cH^2(\sR^{kN})$, 
player $i$ considers the  following objective function (cf.~\eqref{eq: cost_ex_intro}): 
\begin{align}
\label{eq: cost_ex}
J_i(\bu)\coloneqq\sE\left[\int_0^T \left( 
 \ell_i( u_{i,t}) + \frac{1}{N -1 } \sum_{j=1, j\ne i}^N q_{ij}  K( X^{u_i}_{i,t} -X^{u_j}_{j,t}) \right)\,\d t+ g_i (X^{u_i}_{i,T})\right],
\end{align}
where 
for each $i\in [N]$, player $i$'s state process $X^{u_i}_i$ is governed by   the dynamics \eqref{eq:dynamics_ex_intro}, recalled below:
\begin{equation}
\label{eq: dynamics_ex}
\begin{aligned}
      \d X_{i,t}
      &= b_i(t) u_{i,t}  \d t + \sigma_i(t) \d W_t + 
      \sum_{j=1}^m\int_{\sR_0^p}\gamma_{ij}(t, z) \widetilde {\eta}_{j}(\d t, \d z),
      \quad 
      t\in (0,T];\quad X_{i,0} = x_i,
\end{aligned}
\end{equation}
$\ell_i :   \sR^k\to \sR$,  $K: \sR^d\to \sR$,
$g_i: \sR^d\to \sR$
are given measurable functions, 
and 
$q_{ij} \ge 0$ is  a given constant.
Player $i$ aims to minimize \eqref{eq: cost_ex} over the control set (see also  \eqref{eq:admissible_control}):
\begin{equation}
\label{eq:control_ex}
\cA_i=\{u: \Omega\times [0,T]\to A_i\mid 
u \in \cH^2(\sR^k), \|u\|_{\cH^2(\sR^k)}\le U  \},
    \end{equation}
 where $U>0$ is a  sufficiently large constant.

In this game,
each player aims to reach their respective destination, specified by the terminal costs $(g_i)_{i\in [N]}$, at a given terminal time, with their preferred route influenced by the spatial distribution of the population through the kernel $K$ and the interaction weights $(q_{ij})_{i,j\in [N]}$.
Depending on the structure of the kernel   $K$,
the game 
can model self-organizing behavior (commonly referred to as flocking), or aversion behavior,
as discussed in detail below.

\begin{example}
[Kernel choices]
\label{ex:kernel}
When $K$
 decreases 
  as the distance between players increases, the game models congestion-averse behavior, such as pedestrians avoiding densely populated areas.
 One such choice is  the Gaussian-type kernel 
 \begin{equation}
 \label{eq:exponential_kernel}
 K(z)=\exp\left(-\rho{|z|^2}\right), \quad \textnormal{with $\rho>0$},
 \end{equation}
 analogous to the exponentially decaying repulsion function used in collision-avoidance pedestrian models \cite{tordeux2016collision}. An alternative kernel is the following smoothed indicator function:
\begin{equation}
\label{eq:indicator}
    K(z) \coloneqq \int_{\sR^d}\mathbf{1}_{B_r}(z-v)\gamma_\delta (v) \d v,
\end{equation}
where 
\(
\gamma_\delta(v) \coloneqq  \frac{1}{\delta} \gamma\left( \frac{v}{\delta} \right)
\) is a  radially symmetric mollifier, 
with $\gamma:\sR^d\to \sR$ being
  a smooth   function  with compact support, 
and 
$\mathbf{1}_{B_r}$ is the
indicator of 
the ball $B_r$ 
 centered at $0$ with radius $r>0$. 
This kernel function \eqref{eq:indicator} has been used in the nonlocal aversion model \cite{aurell2018mean}, which captures the phenomenon that each pedestrian is only affected by crowding within their personal space $B_r$.  

When $K$
 increases with the distance between players, the model promotes aggregation, mimicking coordinated motion in flocks or herds, which is   driven by factors such as safety, energy efficiency, or social alignment. 
To model such a  self-organizing behavior, one may   use the following quadratic kernel as in \cite{guo2024alpha}: $$K(z)=\frac{1}{2}|z|^2,
$$
or the Cucker–Smale-type flocking kernel used in \cite{santambrogio2021cucker}.

\end{example}
 \subsection{Quantifying $\alpha$}
We impose the following regularity conditions on the model  coefficients.

\begin{assumption}
\label{assum:regularity_crowd_motion} 
For all $i,j \in [N]$, the set $A_i$ and the functions 
$b_i,\sigma_i$ and $  \gamma_{ij}$ satisfy (H.\ref{assum:regularity_general}\ref{item:state_coefficient}). 
The functions 
$(\ell_i)_{i\in [N]}$, $K$ and $(g_i)_{i\in [N]}$ 
are twice continuously differentiable with   bounded second-order derivatives,
and $K(z)=K(-z)$ for all $z\in \sR^d$. 

\end{assumption}

Note that all kernel functions specified in   Example
\ref{ex:kernel}
satisfy the regularity conditions in (H.\ref{assum:regularity_crowd_motion}).

The following theorem specializes 
Theorem \ref{thm:potential_diff} to the above crowd motion game.

 \begin{theorem}\label{thm:crowd_alpha_general}
     Suppose  (H.\ref{assum:regularity_crowd_motion}) holds.
     Let $B=\max_{i\in [N]}\|b_i\|_{L^2}$,
and $\kappa =\|\partial^2_{xx} K\|_{L^\infty}$.
     The crowd motion game defined by \eqref{eq: cost_ex}-\eqref{eq: dynamics_ex}
 is an $\alpha_N$-potential game with 
\begin{equation}
\label{eq:alpha_N}
        \alpha_N \leq T B^2 U^2 \frac{\kappa}{N-1}  \max_{i\in [N]}\sum_{j\ne i} |q_{ji}- q_{ij}|.
\end{equation}

 \end{theorem}

The upper bound of $\alpha_N$ in 
\eqref{eq:alpha_N}
 characterizes the degree of asymmetric  interactions between any two players in the dynamic game \eqref{eq: cost_ex}–\eqref{eq: dynamics_ex}, 
expressed in terms of  the time horizon, the curvature of the kernel $K$ and the  interaction weights $(q_{ij})_{i,j\in [N]}$.
Note that the curvature $\kappa$ can, in turn, be bounded by the parameter $\rho$ in the exponential interaction kernel \eqref{eq:exponential_kernel}, and by the parameter $r > 0$ in the smoothed indicator kernel \eqref{eq:indicator}. These parameters quantify the sensitivity of each player to the distance of other players.

To derive a more explicit bound on $\alpha_N$, we impose additional structure on the interaction weights as follows.

\begin{enumerate}[(a)]
    \item
    \label{item:symmetric}
\textbf{Symmetric interaction.} 
The   weights \((q_{ij})_{i,j \in { [N]}} \) satisfy the \emph{pairwise} symmetry condition 
\begin{equation}\label{eq:qij_symm}
    q_{ij} = q_{ji}, \quad \forall  i,j\in { [N]}.
\end{equation}
This symmetry condition is satisfied   when \eqref{eq: cost_ex} involves mean field interactions (i.e., $q_{ij}=1$)
\cite{aurell2018mean, carmona2018probabilistic, santambrogio2021cucker}, and more generally when the weights are derived from a symmetric graph, as in graphon mean field games (see e.g., 
\cite{aurell2022stochastic}).  
 \item \label{item:asymmetric}\textbf{Asymmetric interaction.} 
 To capture asymmetric interactions, we assume that the interaction weights are determined by an underlying undirected graph $G$, where 
 the vertices represent the set of players $[N]$,
 and each edge indicates a connectivity relation between  the corresponding players.
 
 Suppose that 
$G$ has a bounded degree
$\max_{i\in [N]}\deg(i) = {d}_G$
for some 
${d}_G \geq 2$,
i.e., 
each player is connected to at most ${d}_G$ players.
Additionally, we assume that the asymmetry in interactions diminishes as the graph distance between players increases.
 In particular, 
we consider  the case where  the degree of asymmetry  exhibits an exponential decay:
\begin{equation}
\label{eq:exponential_decay}
|q_{ij} - q_{ji}|  \leq w_{i,j} {\rho}^{c(i,j)}, \quad \forall i,j\in[N], i\ne j
\end{equation}
where $(w_{i,j})_{i,j\in {[N]}}$ are  distinct
   positive constants that are uniformly bounded in 
$N$,   $\rho\in (0,1)$
is a given constant,
and 
$  c(i,j)$ is   the 
(shortest-path) distance between vertices $i$ and $j$.
We also consider the case where    the degree of asymmetry 
 exhibits a polynomial   decay:  
\begin{equation}\label{eq:power_decay}
    |q_{ij} - q_{ji}|\leq w_{i,j}   \frac{1}{c(i,j)^\beta} , \quad \forall i, j \in { [N]}, i\not =j,
\end{equation}
where  
$\beta>0  $ is a given constant, and  
  $(w_{i,j})_{i,j\in { [N]}}$ are distinct positive constants that are uniformly bounded in $N$.

\end{enumerate}

The following corollary refines 
the upper bound on $\alpha_N$ in Theorem \ref{thm:crowd_alpha_general} for both cases \ref{item:symmetric} and \ref{item:asymmetric}, 
providing an explicit dependence on the number of players $N$,
as well as on the parameters $\rho, d_G$ and $\beta$,  which capture the strength and asymmetry of player interactions.

\begin{corollary} \label{cor:crowd_interaction_alpha}
 Suppose  (H.\ref{assum:regularity_crowd_motion}) holds. The crowd motion game defined by \eqref{eq: cost_ex}--\eqref{eq: dynamics_ex} is an $\alpha_N$-potential game with 
 $$
 \alpha_N \leq   \kappa  T B^2 U^2 \zeta_N,
 $$
 where $\kappa$ and   $B $ 
 are defined as in Theorem \ref{thm:crowd_alpha_general}, and   $\zeta_N$ is  determined by the structure of the interaction weights 
 $(q_{ij})_{i,j\in [N]}$ as follows:
\begin{enumerate}[(a)]
    \item
    \label{item:symmetry}
    If $(q_{ij})_{i,j\in [N]}$ satisfies the symmetry condition \eqref{eq:qij_symm}, then $\zeta_N = 0$, i.e.,
    the game is a potential game.
    
    \item 
    \label{item:exponential}
    If $(q_{ij})_{i,j\in [N]}$ satisfies  the exponential decay condition \eqref{eq:exponential_decay}, then as $N\to \infty$,
    \begin{align*}
        \zeta_N =
        \begin{cases}
         \cO \left(N ^{\frac{\ln \rho}{\ln d_G}}\right), &\textnormal{if } \rho \in (1/d_G, 1), 
         \\
           \cO\left( \dfrac{\ln N}{N  }\right), & \textnormal{if }
           \rho = 1/d_G, \\[6pt]
           \cO\left(   {N}^{-1} \right), & \textnormal{if } \rho \in (0, 1/d_G).
        \end{cases}
    \end{align*}
    
    \item
    \label{item:polynomial}
    If $(q_{ij})_{i,j\in [N]}$ satisfies the power-law decay condition \eqref{eq:power_decay}, then as $N\to \infty$,
    \begin{align*}
        \zeta_N =  
         \mathcal{O}\left( \dfrac{\ln \ln N }{
        \left(\ln N\right)^\beta} \right).
    \end{align*}
\end{enumerate}

\end{corollary}

\begin{proof}
Let 
$\zeta_N = \frac{1}{N-1}  \max_{i\in [N]}\sum_{j\ne i} |q_{ji}- q_{ij}|$.
It is clear that 
      $\zeta_N =0$ under Condition \eqref{eq:symmetry_condition}, which proves Item \ref{item:symmetry}. 
      To prove Items \ref{item:exponential} and \ref{item:polynomial},
we assume without loss of generality that 
for all $i,j\in [N]$ with $i\not =j$, $c(i,j)<\infty$, since otherwise $|q_{ij}-q_{ji}|=0$ under 
Condition  \eqref{eq:exponential_decay} or Condition \eqref{eq:power_decay}. 

We first introduce the following  rebalancing technique for the underlying graph $G$: 
Fix node $i \in [N]$. Let $T_1 \subset G$ be the tree with node $i$ as its root.  $T_1$ contains the shortest path for each $j\ne i $ to the root $i$,
and denote  $c_1$ by the shortest-path distance in $T_1$, which satisfies  
$$c_1(i,j) = c(i,j), \quad \forall j \ne i. $$
We will
rebalance the tree $T_1$
   as follows to obtain a $d_G$-ary tree $T_2$, in which every node except those at the deepest level has exactly $d_G$ children:  starting from a node \( j \) that is farthest from the root, we traverse the tree (e.g., depth-first search or breadth-first search)  to move \( j \) to a higher level that is available, reducing its distance to the root \( i \).
We repeat this process until 
  no further adjustment can be made.   We denote $L+1$ as the number of levels in $T_2.$ Specifically, $L$ is the smallest integer that $1 + d_G + d_G^2 + \cdots + d_G^L \geq N $. So as   $N\to \infty$,  $N = \cO(d_G^L)$ and $L = \cO(\frac{\ln N}{\ln d_G}).$
Let $c_2$ denote the distance in $T_2$. Since the reblancing process shortens the distance between the nodes, $$c_2(i,j) \leq c_1(i,j) = c(i,j),
\quad j\not =i.$$

For Item \ref{item:exponential}, there exists a constant $C\ge 0$, which depends only on $(w_{ij})_{i,j\in [N]}$
and $d_G$, such that
\begin{align}\label{eq:bound_exp}
	   \zeta_N \leq \frac{C}{N} \max_{i \in [N]} \sum_{j \ne i} \rho^{c(i,j)}
   \leq \frac{C}{N} \sum_{\ell=1}^{L} \rho^{\ell} d_G^\ell,
\end{align}
where the first inequality follows from
  Condition  \eqref{eq:exponential_decay},
and 
 the last inequality uses $\rho\in (0,1)$ and   the rebalancing technique, which is an upper bound of the summation of weights in $T_2$.
 It remains to compute the right-hand side of \eqref{eq:bound_exp}.
If  $\rho d_G = 1$,
\begin{equation}\label{eq:exp_bound1}
    \zeta_N \le \frac{C L}{N} = \cO\left(\frac{\ln N}{N \ln d_G}\right).
\end{equation}
If $\rho d_G \ne 1$,
\begin{equation}\label{eq:exp_bound2}
    \zeta_N \le \begin{cases} C \frac{1}{d_G^L}   \rho d_G   \frac{(\rho d_G)^L - 1}{\rho d_G - 1}    \le C   \rho^L  = \cO \left(N ^{\frac{\ln \rho}{\ln d_G}}\right), &  \textnormal{if } \rho d_G >1, \\
         \frac{C}{N}\frac{\rho d_G}{1-\rho d_G}=\cO\left(N^{-1}\right), & \textnormal{if }   \rho  d_G < 1.
     \end{cases}
\end{equation}
Combining \eqref{eq:exp_bound1} and \eqref{eq:exp_bound2} finishes the proof for Item \ref{item:exponential}.

For Item \ref{item:polynomial},
fix $i\in [N]$,
 let \( n_\ell \) denote the number of nodes at distance \( \ell \) from the root in $T_1$. 
Then
under  Condition \eqref{eq:power_decay},
\begin{equation}
\label{eq:zeta_N_power}
\zeta_N\le \frac{2}{N}
\sum_{j\not =i} |q_{ij}-q_{ji}| \le  
\frac{2}{N}
\left(\max_{i,j\in [N]}|w_{ij}|\right)
\sum_{\ell=1}^N \frac{n_\ell}{\ell^\beta} \leq 2\left(\max_{i,j\in [N]}|w_{ij}| 
\right)\frac{1}{N}\sum_{\ell=1}^L \frac{d_G^\ell}{\ell^\beta},
\end{equation}  
where the last inequality provides 
an upper bound of $\sum_{\ell=1}^N \frac{n_\ell}{\ell^\beta}$ using the rebalanced tree $T_2$.
Observe that the function
$h(x)\coloneqq d_G^x/x^\beta$ 
has the derivative 
$h'(x)=\frac{d_G^x (x \ln d_G-\beta )}{x^{\beta+1}}$,
and   
is increasing on $(\beta/\ln d_G,\infty)$.
Hence for all $M\in \{1,\ldots, L\}$ with $M\ge \beta/\ln d_G$, 
\begin{equation}\label{eq:sum_power}
    \sum_{\ell =1}^L \frac{(d_G)^\ell}{\ell ^\beta } \leq \sum_{\ell = 1}^{L-M} \frac{(d_G)^\ell}{\ell ^\beta } + M  \frac{(d_G)^L}{L^\beta } 
\end{equation}
Since $x\to 1/x^\beta$ is deceasing on $(0,\infty)$, 
the first term on the right-hand side of \eqref{eq:sum_power}
can be upper bounded by 
$$\sum_{\ell=1}^{L-M} \frac{1}{ \ell^\beta} \leq 1 + \int_1^{L-M} \frac{1}{x^\beta} \d x =\begin{cases}
 1+ \frac{1}{\beta-1} \left(1- {(L-M)}^{(1-\beta)}\right), & \text{ if } \beta > 1, \\
  1+ \ln (L-M), & \text{ if } \beta = 1,
  \\
  1+ \frac{1}{1-\beta} \left( {(L-M)}^{(1-\beta)}-1\right), & \text{ if } 0<\beta < 1.
\end{cases}  $$
Thus for $\beta>1$, 
taking $M^* = \beta\left\lfloor\frac{\ln L}{\ln {d_G}}\right\rfloor$, which implies that $(d_G)^{M^*} = \cO(L^\beta)$ as $L\to \infty$. 
By \eqref{eq:sum_power},
\begin{equation}\label{eq:sum_power_beta>1}
    \sum_{\ell =1}^L \frac{(d_G)^\ell}{\ell ^\beta } \leq 
 C\left(  (d_G)^{L-M^*}   
    +\ln L 
    \frac{(d_G)^L}{L^\beta }
    \right)
    \le C   \ln L 
    \frac{(d_G)^L}{L^\beta }, 
\end{equation}
which along with \eqref{eq:zeta_N_power} shows that as $N\to \infty$,
$$
\zeta_N \le C
\frac{1}{N} \ln  L \frac{(d_G)^L}{L^\beta }
\le C \ln \ln N
\left(\frac{1}{\ln N}\right)^\beta.
$$
For $\beta=1$, 
taking $M^* =  \left\lfloor\frac{\ln L}{\ln {d_G}}\right\rfloor$,
which implies that $(d_G)^{M^*} = \cO(L)$ as $L\to \infty$.
By \eqref{eq:sum_power},
\begin{equation} 
    \sum_{\ell =1}^L \frac{(d_G)^\ell}{\ell ^\beta } 
    \le C\left(  (d_G)^{L-M^*} \ln L 
    +\ln L 
    \frac{(d_G)^L}{L }
    \right)
    \le C   \ln L 
    \frac{(d_G)^L}{L },
\end{equation}
which along with \eqref{eq:zeta_N_power} implies
$\zeta_N = \cO
\left(\ln \ln N
\left(\frac{1}{\ln N}\right)\right)$
  as $N\to \infty$. 
Similarly, 
for $\beta\in (0,1)$,
taking $M^* =  \left\lfloor\frac{\ln L}{\ln {d_G}}\right\rfloor$ and using 
\eqref{eq:sum_power} yield 
\begin{equation} 
    \sum_{\ell =1}^L \frac{(d_G)^\ell}{\ell ^\beta } 
    \le C\left(  (d_G)^{L-M^*} L^{1-\beta}  
    +\ln L 
    \frac{(d_G)^L}{L^\beta }
    \right)
    \le C   \ln L 
    \frac{(d_G)^L}{L^\beta },
\end{equation}
which along with \eqref{eq:zeta_N_power} implies
$\zeta_N = \cO
\left(\ln \ln N
\left(\frac{1}{\ln N}\right)^\beta\right)$
  as $N\to \infty$. This completes the proof. 
\end{proof}

\subsection{Numerical results for NEs}

We apply Algorithm \ref{alg:PG}
to compute the NEs in  the 
crowd motion game  \eqref{eq: cost_ex}--\eqref{eq: dynamics_ex}.  
For ease of exposition,
we consider a four-player game  (i.e., $N=4$), 
where 
each player has two-dimensional state and control processes (i.e., $d=k=2$ and $A_i=\sR^2$).
Player $i$'s state dynamics is given by 
\begin{equation}
\label{eq:state_numerics}
    \begin{aligned}
         \d X_{i,t}^{u_i} 
      &=  u_{i,t}  \d t + \sigma_i  \d W^i_t +  \gamma_i  \d \widetilde \eta_{i,t} 
      + \gamma_0  \d \widetilde \eta_{0,t} 
      , \quad X_{i,0} = x_{i,0},
    \end{aligned}
\end{equation}
where 
$\sigma_i,\gamma_i, \gamma_0\ge 0$
are given constants, 
$W^i$ an $ \tilde \eta_i$
are 
two-dimensional Brownian motion and compensated Poisson processes, respectively, representing the idiosyncratic noise for player $i$, 
and $\tilde \eta_0$ is an independent two-dimensional compensated Poisson process modeling the common noise shared by all players. 
The process $\tilde \eta_i$ has a constant intensity $\lambda_i$, 
with $\lambda_0 = 0.25$, $\lambda_1 =0.3$, and   $\lambda_i = 0.2$ for all $i \geq 2$.
Player $i$ 
considers minimizing the objective   \eqref{eq: cost_ex} with the 
terminal time $T=1$, and terminal cost 
\begin{equation}\label{eq:g}
    g_i(x) = c_i |x- z_i|^2,
\end{equation}
where 
$c_i>0$ is a given constant,
and
$z_i\in \sR^d$ is the   target that  player $i$ aims to reach at time $T$.
The running cost $\ell_i$, the kernel $K$ and the interaction weights $(q_{ij})_{i,j=1}^N$ will be specified below. 
Algorithm~\ref{alg:PG} is implemented using neural network-based policies, with the detailed architecture and training procedures described in Appendix~\ref{sec:implmentation}.

\subsubsection{Aversion Games with Idiosyncratic Noises.}

  We first consider a crowd-aversion game in which all players are subject only to idiosyncratic noise. Specifically, we set $\sigma_i = 0.1(i-1)/N$, {$\gamma_i =0.1 $}, and $\gamma_0 = 0$ in \eqref{eq:state_numerics}. All players start from the same initial location  $x_{i,0} = (0,0)$, and aim to reach a common terminal location  $z_i = (0.5, 0.5)$. The terminal cost function $g_i$ is given by \eqref{eq:g}  with   $c_i = 1$, and the running cost $\ell_i$ on control is   $\ell_i(a) = \frac{0.1}{2} |a|^2$.
To model crowd-aversion effects, we adopt the Gaussian kernel $K(z) = 100 \exp(-100 |z|^2)$, and assume uniform interaction weights $q_{ij} = 1$ in \eqref{eq: cost_ex}, representing symmetric aversion among all players.
The resulting crowd motion game  is a potential game 
as shown in Corollary \ref{cor:crowd_interaction_alpha}.

 Figure~\ref{fig:aversion} illustrates the equilibrium trajectories of the players,
 where positions at times $t = 0.25, 0.5, 0.75$ are marked by symbols 1, 2, and 3, respectively. The left panel shows the mean positions computed over $500$ sample trajectories, while the right panel presents a representative single-sample trajectory.
 
All players begin at the same initial location (indicated by a red circle at position $(0, 0)$) and aim to reach a common target (marked by a red cross at $(0.5, 0.5)$). Early in the game, players disperse in different directions to reduce crowding, a behavior induced by the pairwise aversion term in the cost function. Notably, Player 4 takes a wide detour to avoid other players before converging near the destination. The group exhibits loose coordination: although all players share the same goal, their individual trajectories reflect mutual avoidance dynamics.

\begin{figure}[!ht]
    \centering
     \includegraphics[width=0.50\linewidth]{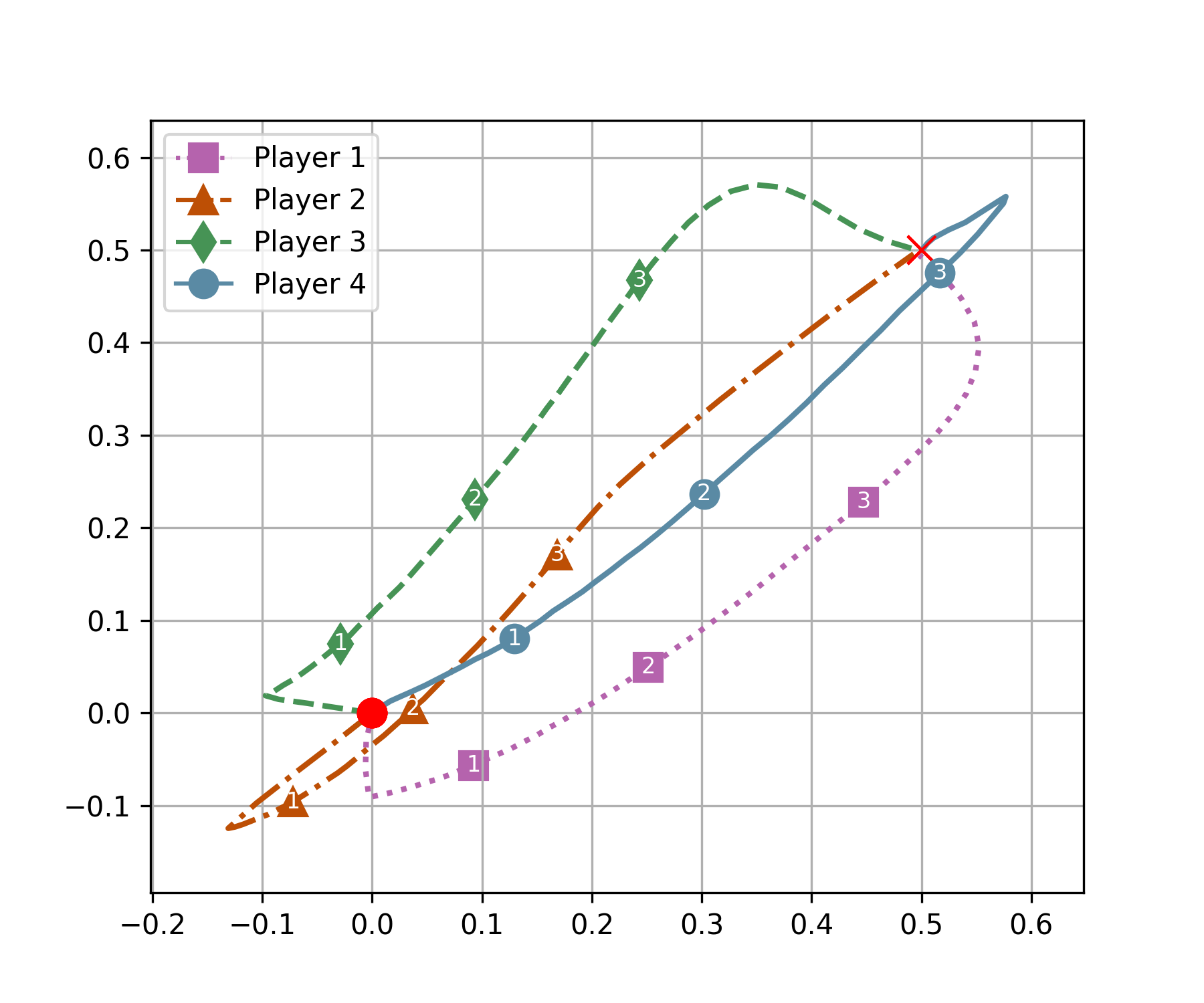}
     \hspace{-10mm}
     \includegraphics[width=0.50\linewidth]{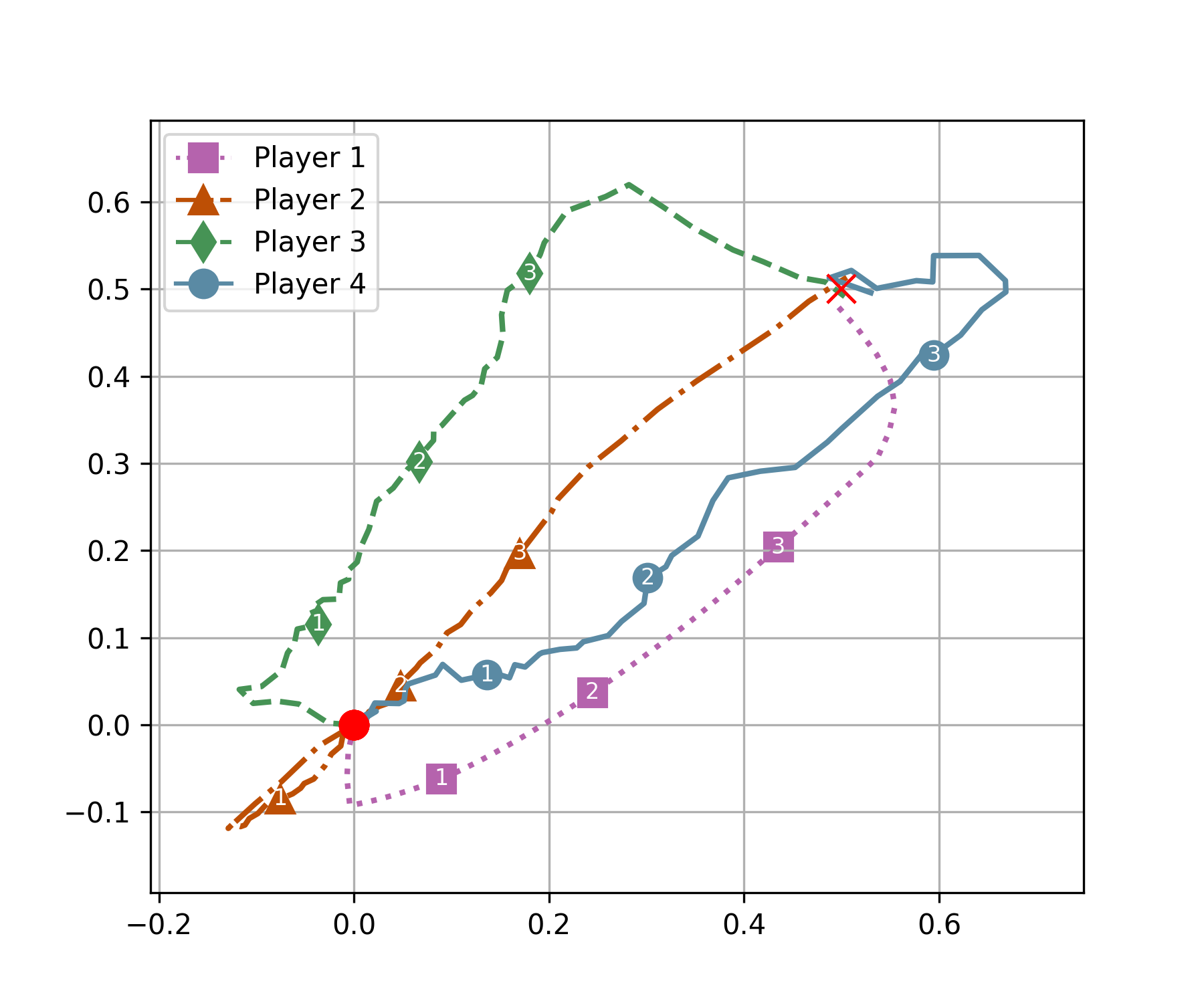}
    \caption{Equilibrium trajectories in the aversion game with a Gaussian kernel  and uniform interaction weights. Left: mean positions over 500 simulations. Right: one representative trajectory. The solid circle denotes the shared initial position; the cross marks the common target. Markers ``1", ``2", and ``3" indicate positions at times $0.25$, $0.5$, and $0.75$, respectively.}
    \label{fig:aversion}
\end{figure}

\subsubsection{Flocking Games with Idiosyncratic Noises.}

The second example considers a flocking game where all players start from the same initial location $x_{i,0} = (0,0)$, and aim for distinct individual target: $(0.25, 0)$, $(0, 0.5)$, $(-0.5, 0)$, and $(0, -1)$. 
Each player is influenced only by idiosyncratic noise, with parameters set as
$\sigma_i=0.1 (i-1)/N $, 
$\gamma_i =0.1$, and $\gamma_0 = 0$ in \eqref{eq:state_numerics}.
The flocking behavior is modeled using the quadratic kernel  $K(z) = \frac{1}{2} |z|^2$. Each player $i$ incurs a running cost on control given by $\ell_i(a) = \frac{0.1}{2} |a|^2$, and a terminal cost defined by \eqref{eq:g}, with $c_i = 40$.

We consider two different settings for the interaction weights $(q_{ij})_{i,j=1}^N$ in \eqref{eq: cost_ex}. In the first setting, uniform interaction is assumed, with $q_{ij} = 1$ for all $i \neq j$, so that each player is equally influenced by every other player. In the second setting, a two-group structure is imposed: players 2 and 3 form one group, and players 1 and 4 form another. In this case, $q_{ij} = 1$ if players $i$ and $j$ belong to the same group, and $q_{ij} = 0$ otherwise. This models selective flocking behavior, where players tend to coordinate only with those in their own group.

Figure~\ref{fig:flock_avg} shows the equilibrium trajectories under uniform interaction weights. In this case, the group first aggregates toward a common intermediate point and, after time $t=0.5$, the players begin to diverge toward their individual destinations. In contrast, Figure~\ref{fig:flock_two} presents the equilibrium trajectories under the two-group interaction structure. Here, each subgroup converges toward a distinct intermediate point, illustrating that the interaction structure encoded in $(q_{ij})_{i,j=1}^N$
   has a significant impact on both the alignment dynamics and the overall configuration of the players.

\begin{figure}[!ht]
    \centering
     \includegraphics[width=0.50\linewidth]{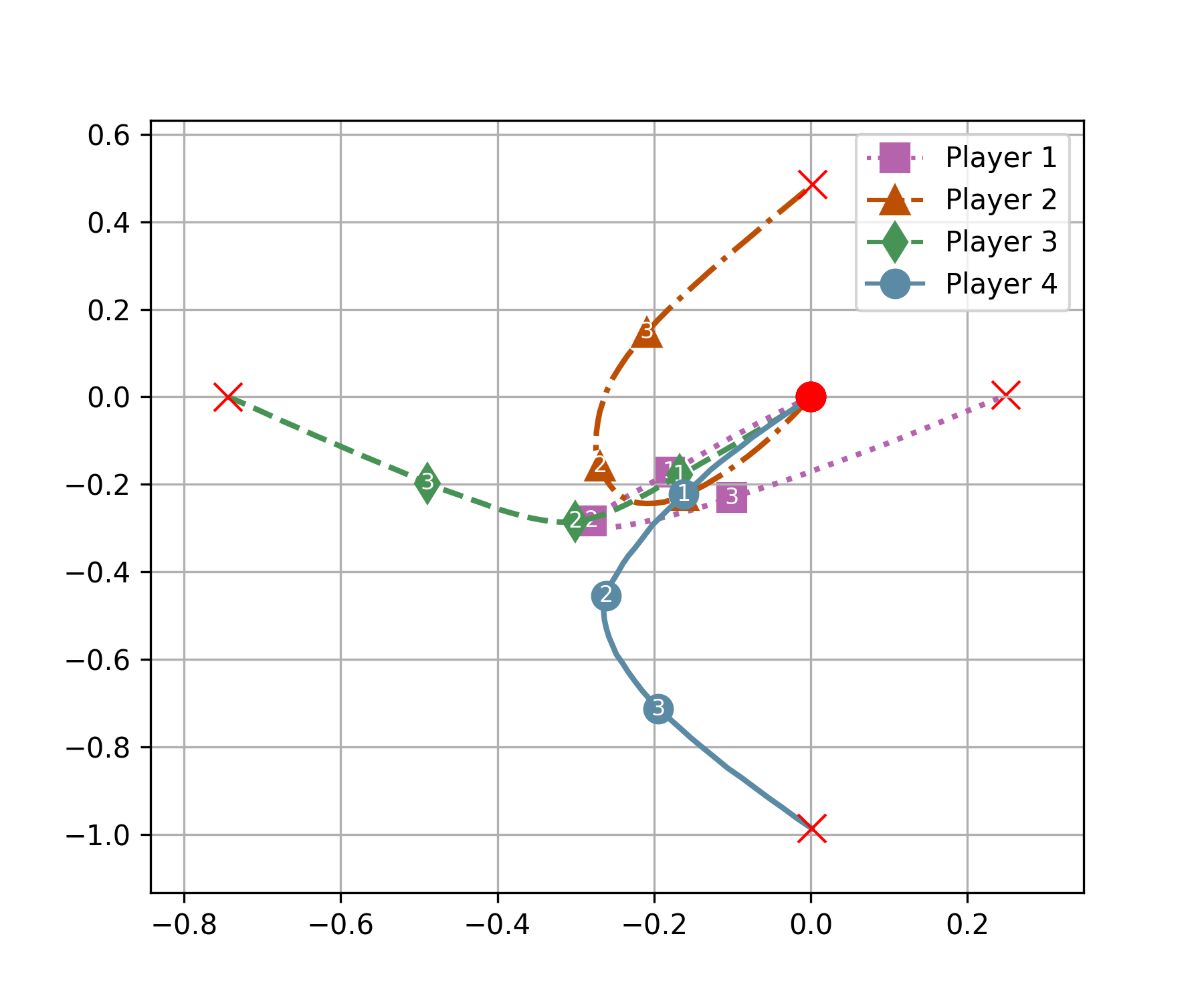}
         \hspace{-10mm}
       \includegraphics[width=0.50\linewidth]{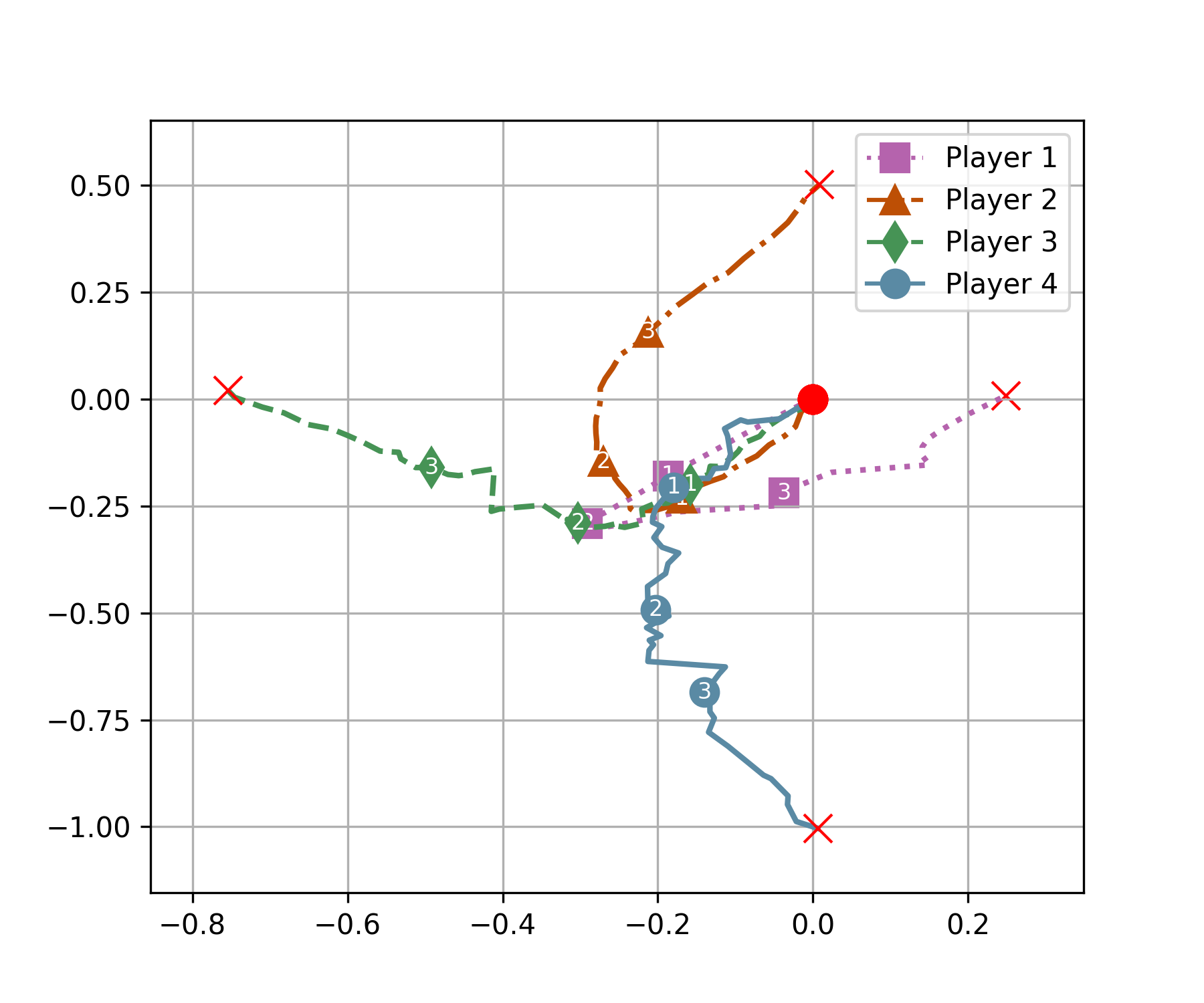}
                \caption{Equilibrium trajectories in the flocking game with a quadratic kernel and uniform interaction weights. Left: mean positions over 500 simulations. Right: one representative trajectory. The solid circle denotes the shared initial position; the crosses mark the individual targets. Markers ``1", ``2", and ``3" indicate positions at times 0.25, 0.5, and 0.75, respectively.}
    \label{fig:flock_avg}
 \end{figure}

 \begin{figure}[!ht]
    \includegraphics[width=0.50\linewidth]{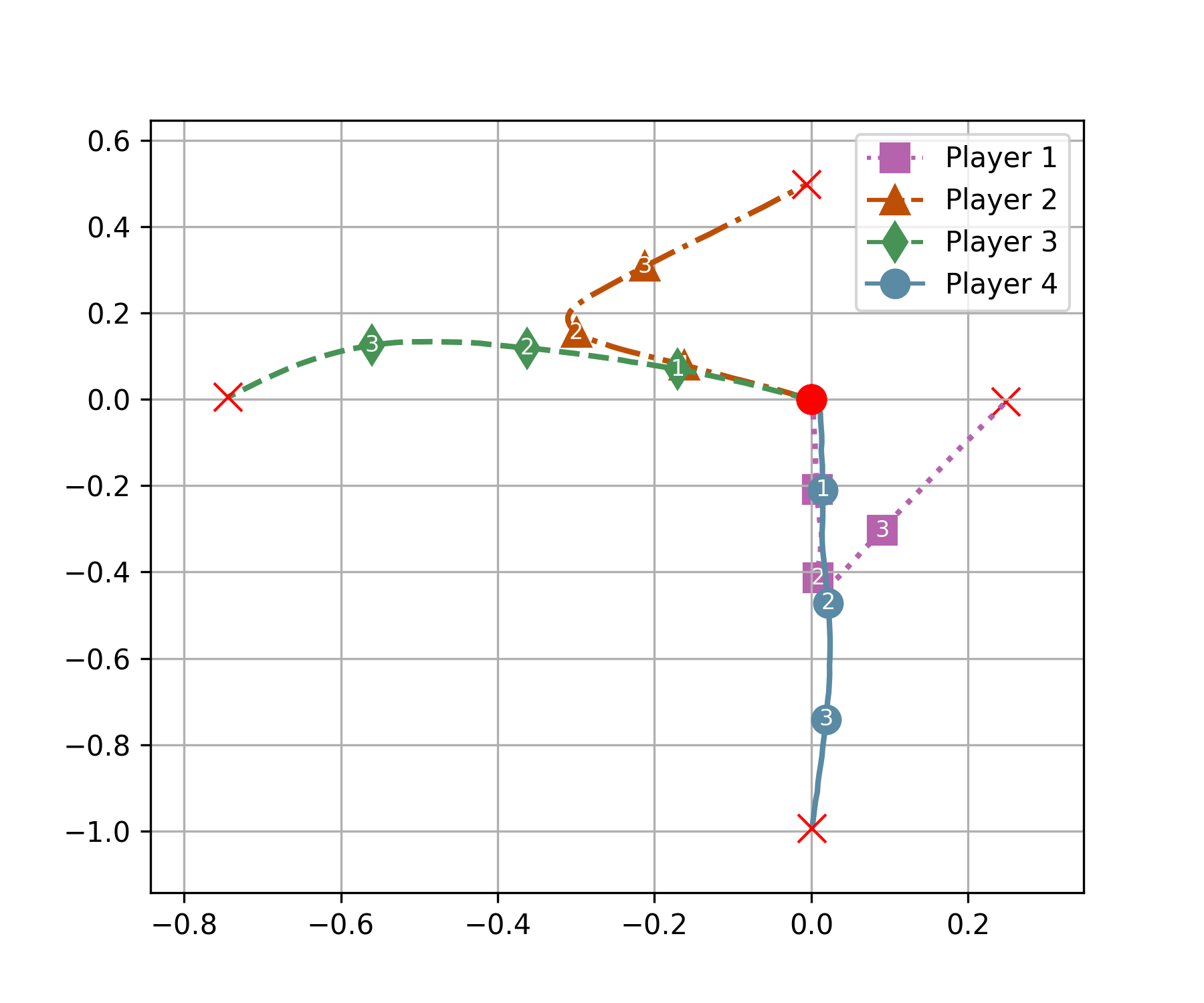}
        \hspace{-10mm}
       \includegraphics[width=0.50\linewidth]{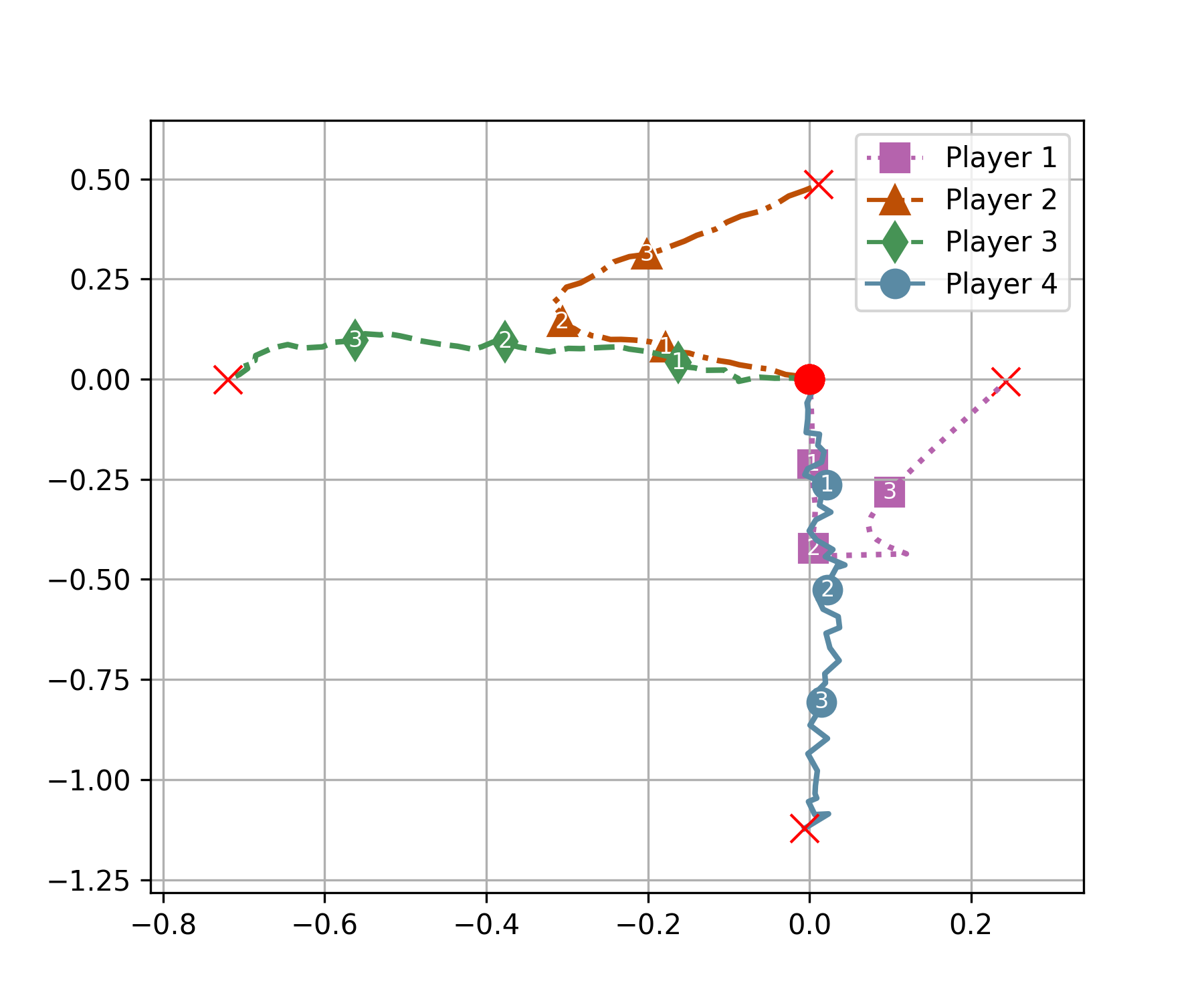}
    \caption{Equilibrium trajectories in the flocking game with a quadratic kernel and group-based interaction weights. Players 1 and 4 belong to one group, and players 2 and 3 form the other. The interaction weights $q_{ij} = 1$  if players $i$ and $j$ are in the same group, and $q_{ij} = 0$ otherwise. Left: mean positions over 500 simulations. Right: one representative trajectory. The solid circle denotes the shared initial position; the crosses mark the individual targets. Markers ``1", ``2", and ``3" indicate positions at times 0.25, 0.5, and 0.75, respectively.}
    \label{fig:flock_two}
\end{figure}

\subsubsection{Flocking Games with Common Noises.}

To demonstrate the flexibility of our framework, we consider a flocking game driven solely by common jumps. Specifically, we set $\sigma_i  =\gamma_i = 0$, and $\gamma_0 = 0.1 $, so that only common noise influences the dynamics. All other model parameters are identical to those in the previous flocking game with uniform interaction weights.

Figure~\ref{fig:flock_common_jump} presents two sample trajectories of the resulting equilibrium dynamics. The common jumps introduce abrupt, synchronized shifts in the players’ positions, followed by realignment as they continue moving toward their respective targets. While the jump events cause irregularities in the intermediate paths, the overall flocking behavior remains consistent with the patterns observed in Figure \ref{fig:flock_avg} for the setting with purely idiosyncratic noise.

\begin{figure}
    \centering
    \includegraphics[width=0.50\linewidth]{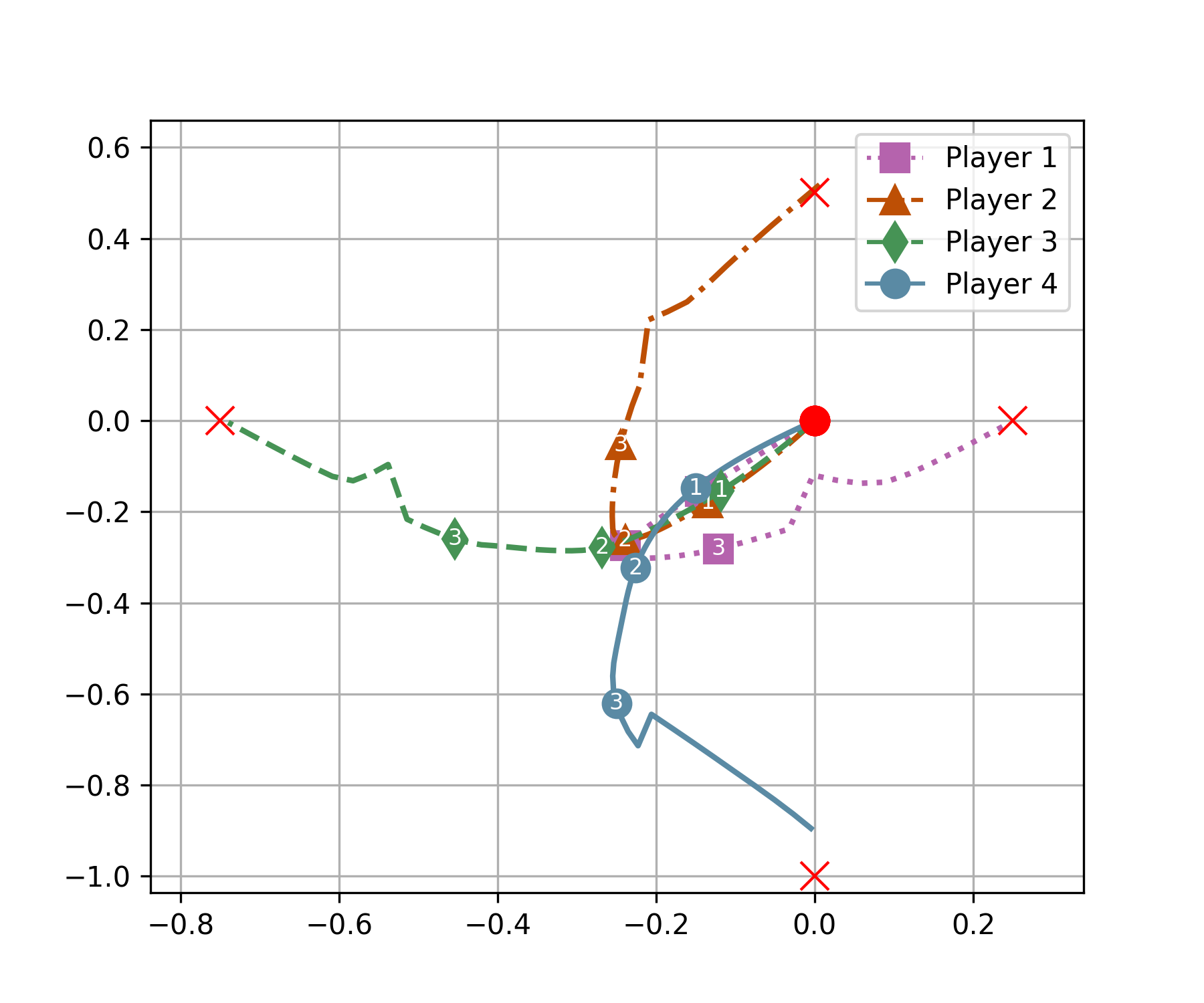}
       \hspace{-10mm} \includegraphics[width=0.50\linewidth]{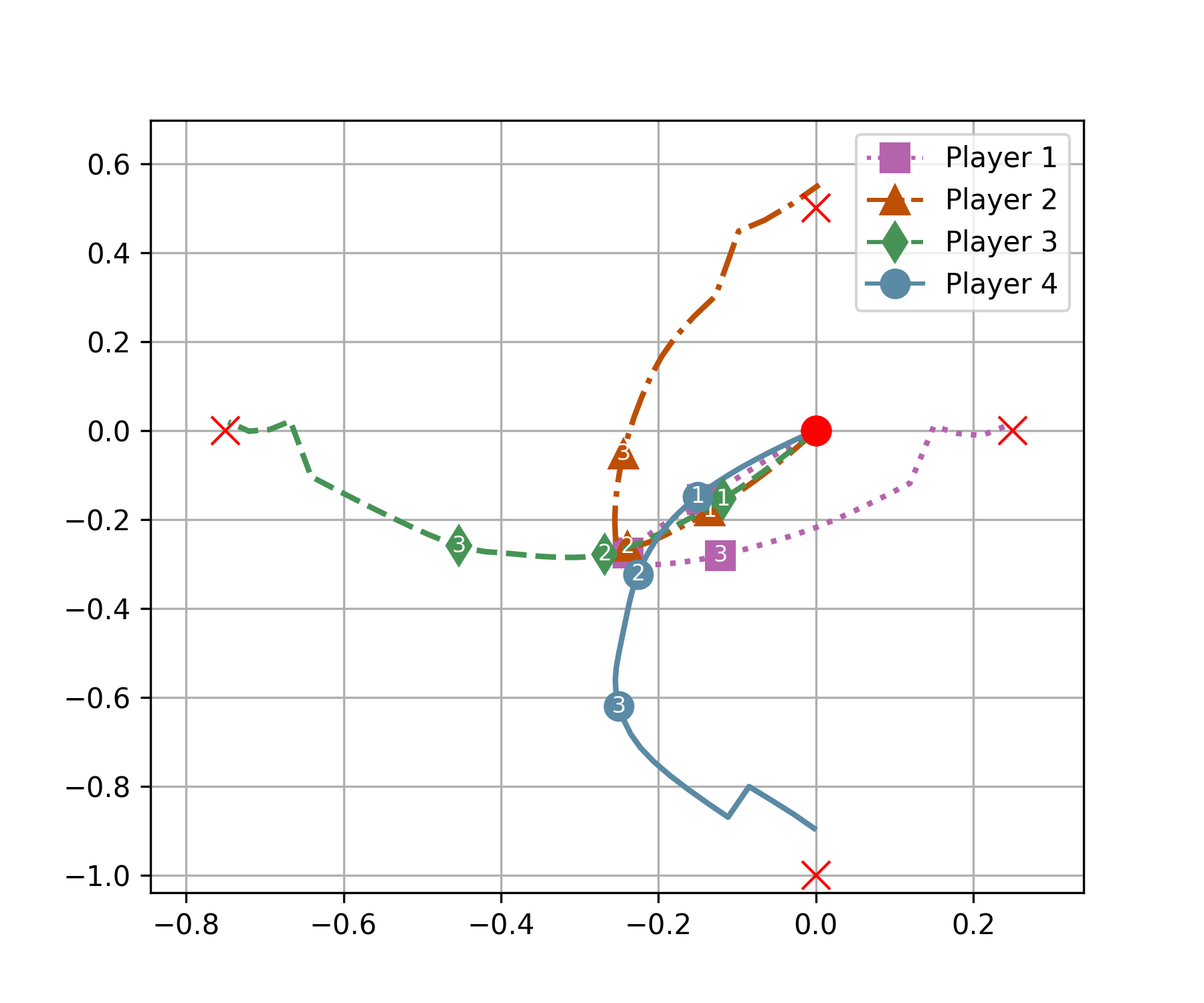}
    \caption{Equilibrium trajectories in the flocking game with a quadratic kernel, uniform interaction weights, and pure common jumps. The solid circle denotes the shared initial position; the crosses mark the individual targets. 
Markers ``1", ``2", and ``3" indicate positions at times 0.25, 0.5, and 0.75, respectively.}
    \label{fig:flock_common_jump}
\end{figure}

{\color{blue}

\section{Application to Mean-Variance Portfolio   Games}

In this section, we exploit our framework to analyze a portfolio selection game under a mean–variance (MV) criterion.
In this game, both the diffusion coefficient and the jump intensity of each player’s state dynamics are controlled, and each player’s objective function depends nonlinearly on the distribution of all players’ states. 
 
\subsection{Problem Setup}
Consider a financial market with $N$ players 
trading in a common risk-free asset  with interest rate $r>0$ and $d$ risky assets,
whose returns are driven by an $m$-dimensional Brownian motion and an independent compensated Poisson random measures $\widetilde{\eta}$, as introduced in Section \ref{sec:math_setup}.
The 
price of the risky stocks follows according to the following  dynamics: 
\begin{equation}
\label{eq:stock}
\mathrm{d} S_t=\operatorname{diag}\left(S_t\right)\left(\mu_t \mathrm{d} t+\sigma_t\mathrm{d} W_t
+\int_{\sR^p_0}\gamma (t,z)\widetilde{\eta}(\d t,\d z)
\right), \quad 
\end{equation}
where  $\operatorname{diag}(x)\in \sR^{d\times d}$ is the  diagonal matrix with diagonal entries $x \in \mathbb{R}^d$,  $
\mu:[0,T]\to \mathbb{R}^d$
and $\sigma:[0,T]\to \mathbb{R}^{d\times m} $ are bounded measurable functions,
and 
$\gamma : [0,T] \times \sR^p \ra \sR^{d }$
is a measurable function satisfying 
$\sup_{(t,z)\in [0,T]\times \sR^p_0} {|\gamma(t,z)|}/{\min(1,  |z|)}<\infty$.

For each $i\in [N]$, the set $\mathcal{A}_i$ of player $i$’s  admissible controls consist of all $\mathbb{F}$-adapted, square-integrable processes
$ 
u_i : \Omega \times [0,T] \to \mathbb{R}^d
$, 
where $u_{i,t}$ represents the number of dollars player $i$ holds in each asset at time $t$.
For each $u_i \in \mathcal{A}_i$, let ${X}_{i}^{u_i}$ denote the \emph{discounted} wealth of player $i$, starting from the initial position $x_i>0$ and following the investment strategy $u_i$. The dynamics of $ {X}_{i}^{u_i}$ are given by 
\begin{equation}
\label{eq:MV_state_original}
\mathrm{d} {X}_{i,t}=u_{i,t}^\top \left((\mu_t -r) \mathrm{d} t+\sigma_t \mathrm{d} W_t 
+ \int_{\sR^p_0}\gamma (t,z)\widetilde{\eta}(\d t,\d z)
\right), \quad t\in (0,T]; \quad {X}_{i,0} =x_i.
\end{equation}
Each player aims to outperform the others in terms of relative wealth, according to an MV preference.
Specifically, 
player $i$ aims  to maximize the following  objective function $J_i:\cA\to \sR$;
\begin{equation}\label{eq:MV_objective_original}
     J_i(\bu) = \gamma_i \E \left[  {{X}^{u_i}_{i,T}}  -  \sum_{j\in [N]\setminus\{i\}} 
    \lambda^M_{ij}  {{X}^{u_j}_{j,T}}\right] - \frac{1}{2} \var \left(
     {{X}^{u_i}_{i,T}} - 
  \sum_{j\in [N]\setminus\{i\}} 
   \lambda^V_{ij} {X}^{u_j}_{j,T}
    \right),
\end{equation}
where $\gamma_i>0$  is  player $i$’s 
 risk aversion parameter,  and $\lambda^M_{ij},\lambda^V_{ij}\ge 0$ specify  player $i$'s
relative preference between their own wealth and other competitors. 
For the sake of exposition, we assume that the weights $(\lambda^V_{ij})_{i,j\in [N]}$ of rank one: for all $i,j\in [N]$,
\begin{equation}
\label{eq:lambda_V}
\lambda^V_{ij}=\theta_i\varphi_j, \quad \textnormal{for some $\theta_i,\varphi_j>0$}.
\end{equation}
Such an interaction structure includes as special cases the MV portfolio selection games analyzed in \cite{shao2025competitive, huang2025partial}, where
  $\varphi_j=1/(N-1)$ for all $j$,
  and $\lambda^M_{ij} =\theta_i/(N-1)$ for all $i,j\in [N]$.

\subsection{Potential Function and NE}
An NE of the game \eqref{eq:MV_state_original}–\eqref{eq:MV_objective_original} is  defined analogously to Definition \ref{def:epsilon_NE}.
To construct a potential function for 
\eqref{eq:MV_state_original}–\eqref{eq:MV_objective_original},
we observe that positive scaling of the objective and shifting the state dynamics by the (constant) initial conditions preserve the set of NEs.

\begin{lemma}
\label{lemma:MV_rescale}
    A control profile $\bu \in \cA$ is an NE of the game \eqref{eq:MV_state_original}–\eqref{eq:MV_objective_original}
    if and only if it is an NE for the   objectives $(\tilde J_i)_{i\in [N]}$ given by 
    \begin{equation}
    \label{eq:MV_rescale}
       \tilde J_i(\bu)\coloneqq \frac{\varphi_i}{\theta_i}\left(
       \gamma_i \E \left[  {\tilde {X}^{u_i}_{i,T}}  -  \sum_{j\in [N]\setminus\{i\}} 
    \lambda^M_{ij}  {\tilde {X}^{u_j}_{j,T}}\right] - \frac{1}{2} \var \left(
     {\tilde {X}^{u_i}_{i,T}} - 
  \theta_i\sum_{j\in [N]\setminus\{i\}} 
  \varphi_j
    \tilde {X}^{u_j}_{j,T}
    \right)\right),
    \end{equation}
    where 
    $ \tilde X_i^{u_i} \coloneqq X_i^{u_i}-x_i$ for all $i\in [N]$.
\end{lemma}

We now show that the rescaled objectives \eqref{eq:MV_rescale} form a potential game.
The proof exploits that 
the shifted state process $\tilde X^{u_i}_i$
and its sensitivity process satisfy a separation principle \eqref{eq:mv_separation} analogous to Lemma \ref{lemma:state_decomposition}, which allows us to construct a simplified potential function with finite-dimensional state processes.  


\begin{prop}
\label{prop:potential_mv}
    The game with  the   objectives $( J_i)_{i\in [N]}$ given in  \eqref{eq:MV_rescale}
    is a potential game with the potential function
    \begin{equation}
        \begin{aligned}
    \Phi(\bu) =  \beta^\top \sE \left[\tilde {\bX}^\bu_T \right] - 
    \frac{1}{2}  \sE \left[(\tilde {\bX}^\bu_T)^\top \Psi \tilde  {\bX}^\bu_T \right] + \frac{1}{2} \sE\left[\tilde {\bX}^\bu_T\right]^\top \Psi \sE\left[\tilde {\bX}^\bu_T\right],
    \end{aligned}
\end{equation}
 where 
 $\tilde{ \bX}^{\bu}=(\tilde  X_i^{u_i} )_{i\in [N]}$
 is given in Lemma \ref{lemma:MV_rescale},
 and 
 $\beta\in \sR^N$ and $\Psi\in \sR^{N\times N}$
 satisfy  
 for all $i\in [N]$,
\begin{equation}
\label{eq:MV_weight}
  \beta_i = \frac{\varphi_i}{\theta_i}\gamma_i,
  \quad 
  \Psi_{ij}=\begin{cases}
   \frac{\varphi_i}{\theta_i}, & j=i,
   \\
   -\varphi_i \varphi_j,
   &j\in [N]\setminus \{i\}.
  \end{cases}
\end{equation}
  
\end{prop}

\begin{proof}
For all $i
\in [N]$,
by \cite[Equation (3.3)]{guo2024alpha}, one can show that 
the  linear derivative of $  \tilde  J_i$ with respect to $u_i$ is given by 
\begin{align}\label{eq:dJ_du}
    \frac{\delta \tilde  {J}_i}{\delta u_i}(\bu; u_i') = 
    \frac{\varphi_i}{\theta_i}
    \left(
    \gamma_i \E\left[Y_{i,T}^{u_i'}\right] -  \sE \left[  
    \left(
     {\tilde {X}^{u_i}_{i,T}} - 
  \theta_i\sum_{j\in [N]\setminus\{i\}} 
  \varphi_j
    \tilde  {X}^{u_j}_{j,T}
    \right)  \left(Y_{i,T}^{u_i'} - \sE\left[ Y_{i,T}^{u_i'}\right]\right)\right]
    \right),
\end{align}
where for all $u'_i\in \cA_i$,
\begin{align} 
\label{eq:Y_h_i_open}
\begin{split}
    \d {Y}^{u'_i}_{i,t}
    &=
  (u'_{i,t})^\top \left( (\mu_t -r) \d t + \sigma \d W_t
  +\int_{\sR^p_0}\gamma (t,z)\widetilde{\eta}(\d t,\d z)\right),
  \quad t\in (0,T];
    \quad {Y}^{u_i'}_{0,i}=0.
 \end{split}
\end{align} 
Moreover, for all $i,j\in [N]$ with $i\not =j$,
$u'_i\in \cA_i$, and 
$u''_j\in \cA_j$,
\begin{align}\label{eq:dJ_du}
    \frac{\delta^2  \tilde  J_i}{\delta u_i \delta{u_j}}
(\bu; u_i', u_j'') = 
    {\varphi_i  \varphi_j} 
  \left   ( \sE \left[  
     Y_{i,T}^{u_i'} {Y}^{u''_j}_{j,T}
    \right] - \sE \left[ Y_{i,T}^{u_i'} \right]
    \sE \left[  {Y}^{u''_j}_{j,T}  \right]
    \right) .
\end{align}
Note that $\frac{\delta^2 \tilde  J_i}{\delta u_i \delta{u_j}}
(\bu; u_i', u_j'') =\frac{\delta^2 \tilde  J_j}{\delta u_i \delta{u_j}}
(\bu; u_i', u_j'') $, which along with \cite[Theorem 2.5]{guo2024alpha} implies that 
the   objectives $(\tilde J_i)_{i\in [N]}$ given in  \eqref{eq:MV_rescale}
    forms  a potential game with a potential function $\Phi:\cA\to \sR$
given  by (cf.~\eqref{eq:potential_general}):
 \begin{align}
 \label{eq:mv_potential_1}
    \Phi(\bu)&=\int_0^1 \sum_{i=1}^N   \frac{\delta \tilde  J_i}{\delta u_i}\left(s\bu ; u_i \right) \d s.
\end{align}
To   simplify the   expression \eqref{eq:mv_potential_1}  of the potential function $\Phi$,
observe that for all $i\in [N]$ and  $u_i\in\cA_i $,
\begin{equation}
\label{eq:mv_separation}
\tilde {X}^{u_i}_{i,t}={Y}^{u_i}_{i,t}, \quad  
 \tilde {X}^{s u_i}_{i,t}=s\tilde {X}^{ u_i}_{i,t},\quad \forall s\in [0,1].
 \end{equation}
 Substituting 
 \eqref{eq:mv_potential_1},
we obtain
\begin{equation}\label{eq:potential_MV}
    \begin{aligned}
            \Phi(\bu)
            & =\sum_{i=1}^N 
             \frac{\varphi_i}{\theta_i}
    \left(
    \gamma_i \E\left[\tilde {X}^{u_i}_{i,T}\right] -  \frac{1}{2}\sE \left[  
    \left(
     {\tilde {X}^{u_i}_{i,T}} - 
  \theta_i\sum_{j\in [N]\setminus\{i\}} 
  \varphi_j
   \tilde  {X}^{u_j}_{j,T}
    \right) \left( \tilde {X}^{u_i}_{i,T} - \sE\left[ \tilde {X}^{u_i}_{i,T} \right] \right)\right]
    \right)
    \\
    &= \beta^\top \sE \left[\tilde {\bX}^\bu_T \right] - 
    \frac{1}{2}  \sE \left[(\tilde {\bX}^\bu_T)^\top \Psi \tilde {\bX}^\bu_T \right] + \frac{1}{2} \sE\left[\tilde {\bX}^\bu_T\right]^\top \Psi \sE\left[\tilde {\bX}^\bu_T\right].
    \end{aligned}
\end{equation}
\end{proof}

Based on Lemma  
   \ref{lemma:MV_rescale} 
and Proposition \ref{prop:potential_mv},
the following theorem constructs an NE of the portfolio selection game. 
To simplify the notation, we define for each $t\in [0,T]$, 
$$
b_t\coloneqq\mu_t -r,
\quad 
\Sigma_t \coloneqq \sigma_t\sigma_t^\top
+\int_{\sR^p_0}\gamma(t,z)\gamma^\top(t,z) \nu(\d z),
$$
where $\nu$ is the compensator  of the Poisson random measure $\eta$. 
The functions $b$ and $\Sigma$ denote the excess return rate and its covariance matrix of the risky assets, respectively.

\begin{theorem} 
\label{thm:potential_maximizer}
Suppose that $     \sum_{i=1}^N \frac{\varphi_i \theta_i}{1+\varphi_i\theta_i} < 1
$  for $(\varphi_i, \theta_i)_{i\in [N]}$  given in \eqref{eq:lambda_V},
and 
for each $t\in [0,T]$,
 $\Sigma^{-1}_t$ exists and   is  uniformly bounded  in $t$.
Define the function
$  a^*: [0,T]\times \sR^N\times \sR^N \to \sR^{d\times N}$ by
\begin{equation}
\label{eq:MV_control_explicit}
  a^* (t, x,y)=
   -\Sigma^{-1}_t 
 b_t \left( (x  - y  )^\top  -  \beta^\top  \Psi^{-1}  \exp\left(\int_t^T  b_s^\top  \Sigma_s^{-1}   b_s \d s\right)\right),
\end{equation}
where $\beta$ and $\Psi$
are given in \eqref{eq:MV_weight},
and   define 
$\bu^*\in \cA$
by 
$\bu_t^* =a^*(t,  \bX^{*}_t, \sE[\bX^{*}_t])$,
where $\bX^{*}$ satisfies the following dynamics:   
\begin{equation}
\label{eq:mv_optimal_state}
\d \bX_t  = 
a^*(t, \bX_t, \sE[\bX_t]) ^\top \left(  b_t  \, \d t +  \sigma_t \, \d W_t
+\int_{\sR^p_0}\gamma (t,z)\widetilde{\eta}(\d t,\d z)\right),
\quad 
\bX_0=(x_1,\ldots,x_N)^\top.
\end{equation}
Then $\bu^*\in \cA$ is an  NE 
of the game \eqref{eq:MV_state_original}–\eqref{eq:MV_objective_original}.

\end{theorem}
\begin{proof}

We first claim that the condition 
$     \sum_{i=1}^N \frac{\varphi_i \theta_i}{1+\varphi_i\theta_i} < 1
$  
is equivalent to the positive definiteness of   $\Psi \in \sR^{N\times N}$ given in \eqref{eq:MV_weight}. 
Indeed, 
$\Psi=D-{\psi}{\psi}^\top$,
where $D=\operatorname{diag}
( \frac{\varphi_1}{\theta_1}+\varphi^2_1,\ldots, \frac{\varphi_N}{\theta_N}+\varphi^2_N)$
and $\varphi=(\varphi_1,\ldots, \varphi_N)^\top$.
Applying Schur complement
(see e.g., \cite[Section A.5.5]{boyd2004convex})
to the matrix 
$\begin{pmatrix}
    D & \varphi
    \\
    \varphi^\top & 1
\end{pmatrix}$ yields that $\Psi$ is positive definite if and only if $1-\varphi^\top D^{-1} \varphi>0$, which is equivalent to  $     \sum_{i=1}^N \frac{\varphi_i \theta_i}{1+\varphi_i\theta_i} < 1
$.  

By Lemma  \ref{lemma:MV_rescale}
and 
 Proposition \ref{prop:potential_mv},
any maximizer of $\Phi$ is an NE $\bu^*$ of the game \eqref{eq:MV_state_original}–\eqref{eq:MV_objective_original}.
We now construct a maximizer of $\Phi$ through a  verification argument as in  \cite{guo2023ito}.
To this end, we denote by 
$\sS^N$
the space of $N\times N$ symmetric matrices, 
 by $\cP_2(\sR^N)$
the space of probability measures on $\sR^N$ with second moments, 
and   define    for all $\mu \in \cP_2(\sR^N)$ and $M \in \sS^N$,      
$ \bar {\mu}\coloneqq \int_{\sR^N} x\mu (\d x),
$ and $
\sV(\mu)(M)\coloneqq \int_{\sR^N}x^\top M x\mu (\d x)-
\bar \mu^\top M \bar  \mu
$. 

Define the   function
${V}:[0,T]\times \cP_2(\sR^N) \ra \sR$ 
such that for all $(t,\mu)\in[0,T]\times\cP_2(\sR^N)$,
\begin{align*}
    {V}(t,\mu) \coloneqq  - \frac{1}{2} 
    \sV(\mu)(M(t))  + 
    \beta^\top 
    \bar \mu +    c(t).
\end{align*}
where  $M:[0,T]\to\sS^N$ and $c:[0,T]\to\R$ are solutions to the following linear equations:
\begin{equation}
\label{eq:potential_ode_M}
   \begin{aligned}
     &\dot M(t) -   b_t^\top  \Sigma_t^{-1}   b_t M(t) = 0, \quad t\in [0,T);\quad M(T) = \Psi, \\
       & \dot c(t) +    \frac{1}{2}   b_t^\top \Sigma_t^{-1}   b_t \beta^\top M(t) ^{-1} \beta = 0,  \quad t\in [0,T); \quad  c(T) =0.
\end{aligned}
\end{equation}
where the dot refers to the time derivative. 
Note that for all $t\in [0,T]$,
\begin{equation} 
\label{eq:M_solutioin}
    M(t) = \Psi \exp\left(-\int_t^T  b_s^\top  \Sigma_s^{-1}   b_s \d s\right),
\end{equation}
and   hence is positive definite since $\Psi$ is positive definite.
Moreover, a direct computation yields 
\begin{align}
\label{eq:value_derivative_MV}
\begin{split}
    \partial_t V(t,\mu) &= - \frac{1}{2} 
    \sV(\mu)(\dot M(t))  + 
     \dot c(t),\\
     \delta_\mu V(t,\mu)(x) &=
 -\frac{1}{2} x^\top M(t) x 
 +(\bar{\mu}^\top M(t)+\beta^\top ) x
 + C,\\
     \partial_\mu  {V}(t,\mu)(x) 
    &= - M(t) (x-\bar \mu)+ \beta,
    \\ \partial_{x}\partial_{\mu} {V}(t,\mu)(x) &= -{M}(t),
    \end{split}
\end{align}
where 
$\delta_\mu$ and 
$\partial_\mu$ refer to the linear functional derivative and Lions derivative with respect to the measure component, respectively, and 
$C$ is a constant.

         For each $\bu \in \cH^2(\sR^{d\times N})$, 
       applying It\^o's formula 
(e.g.~\cite[Theorem 2.1]{li2018mean}
and \cite[Corollary 3.5]{guo2023ito})
 to the map $t\mapsto V(t, \sP_{\tilde{\bX}_t^{  \bu} })$ yields 
\begin{align}
\label{eq:ito_potential}
    \begin{split}
  &      V(T,   \sP_{\tilde{\bX}_T^{  \bu} } )-V(0,   \sP_{\tilde {\bX}_0^{  \bu} } )
  \\
& =\int_0^T \bigg(
 \partial_t {V} (t,\sP_{\tilde {\bX}_t^{  \bu}}) +\sE  \bigg[     b_t^\top \bu_t 
   \partial_\mu  {V}(t,\sP_{\tilde {\bX}_t^{  \bu}} )(\tilde {\bX}_t^{  \bu})
   + \frac{1}{2} \Tr \left(\bu_t^\top \sigma_t \sigma_t^\top \bu_t 
\partial_{x}\partial_{\mu} {V}
   (t,\sP_{\tilde {\bX}_t^{  \bu}} )(\tilde {\bX}_t^{  \bu})
 \right)
 \\
 &\quad +  \int_{\sR^p_0} \left(
 \delta_\mu V(t,\sP_{\tilde {\bX}_t^{  \bu}})(\tilde {\bX}_t^{  \bu}+\bu^\top_t\gamma(t,z))
 -\delta_\mu V(t,\sP_{\tilde {\bX}_t^{  \bu}})(\tilde {\bX}_t^{  \bu})
 -\partial_\mu V(t,\sP_{\tilde {\bX}_t^{  \bu}})(\tilde {\bX}_t^{  \bu})^\top \bu^\top_t\gamma(t,z)
 \right)\nu(\d z)
 \bigg]
\bigg)\, \d t
\\
&= 
\int_0^T \bigg(
 \partial_t {V} (t,\sP_{\tilde {\bX}_t^{  \bu}}) +\sE  \bigg[  -   b_t^\top \bu_t \left[ M(t) (\tilde {\bX}_t^{  \bu}  - \sE[\tilde {\bX}_t^{  \bu}]  )-  \beta \right]  -  \frac{1}{2} \Tr \left(\bu_t^\top \sigma_t\sigma_t^\top \bu_t M(t) 
 \right)
 \\
 &\quad -\frac{1}{2} \Tr \left(\bu_t^\top \int_{\sR^p_0}\gamma(t,z)\gamma^\top(t,z) \nu(\d z) \bu_t M(t) 
 \right)
 \bigg]
\bigg)\, \d t,
\end{split}
\end{align}
where the last identity used  \eqref{eq:value_derivative_MV}.
 Since $M(t)$ is positive definite, 
 for each $(t,x,y)\in [0,T]\times \sR^N\times \sR^N$,
the function 
$$
\sR^{d\times N} \ni a \mapsto 
H(t,x,y,a)\coloneqq 
-   b_t^\top a \left[ M(t) (x  - y  )-  \beta\right]  -  \frac{1}{2} \Tr \left(a ^\top \Sigma_t^\top a M(t) \right)\in \sR
$$
is maximized at 
\begin{equation}
\label{eq:mv_policy}
a^*(t,x,y)= -\Sigma^{-1}_t 
 b_t \left[ M(t) (x  - y  )-  \beta \right]^\top(M(t))^{-1},
\end{equation}
which along with \eqref{eq:M_solutioin} yields 
\eqref{eq:MV_control_explicit}.
Moreover,  
$$
\sup_{a\in \sR^{d\times N}}
H(t,x,y,a) = \frac{1}{2}  b_t^\top \Sigma^{-1}_t 
 b_t \left[ M(t) (x  - y  )- \beta \right]^\top(M(t))^{-1}  \left[ M(t) (x  - y  )-  \beta \right].
$$
This along with \eqref{eq:ito_potential}
implies that 
\begin{align*}
&      V(T,   \sP_{\tilde {\bX}_T^{  \bu} } )-V(0,   \sP_{\tilde {\bX}_0^{  \bu} } )
\\
&
      \le 
      \int_0^T \left(
 \partial_t {V} (t,\sP_{\tilde {\bX}_t^{  \bu}}) 
 +
  \frac{1}{2}  b_t^\top \Sigma^{-1}_t 
 b_t 
\sE\left[ 
\left[ M(t) (\tilde {\bX}_t^{  \bu}  - \sE[\tilde {\bX}_t^{  \bu}]  )-  \beta\right]^\top(M(t))^{-1}  \left[ M(t) (\tilde {\bX}_t^{  \bu}  - \sE[\tilde {\bX}_t^{  \bu}]    )-  \beta \right] 
\right]
 \right)
\, \d t
\\
&=
    \int_0^T \left(
 \partial_t {V} (t,\sP_{\tilde {\bX}_t^{  \bu}}) 
 +
  \frac{1}{2}  b_t^\top \Sigma^{-1}_t 
 b_t
\left[ \sV(\sP_{\tilde {\bX}_t^{  \bu}}) (M(t))+  \beta^\top (M(t))^{-1} \beta
\right]
 \right)
\, \d t=0,
\end{align*}
where the last identity used the 
fact that $M$ and $c$ satisfy
\eqref{eq:potential_ode_M}. 
This implies that  
$$
\Phi(\bu ) \le V(0,\delta_0), \quad \forall \bu \in \cH^2(\sR^{d\times N}),
$$
and the equality is achieved at the control 
$\bu^*\in \cA$
by 
$\bu_t^* =a^*(t,  \tilde{\bX}^{*}_t, \sE[\tilde{\bX}^{*}_t])$,
where $\tilde{\bX}^{*}$ satisfies the following dynamics:   
\begin{equation*}
\d \tilde{\bX}_t  = 
a^*(t, \tilde{\bX}_t, \sE[\tilde{\bX}_t]) ^\top \left(  b_t  \, \d t +  \sigma_t \, \d W_t
+\int_{\sR^p_0}\gamma (t,z)\widetilde{\eta}(\d t,\d z)\right),
\quad 
\tilde{\bX}_0=0.
\end{equation*}
Observe that
$\tilde{\bX}^*_t={\bX}^*_t-(x_1,\ldots,x_N)^\top$ for $\bX^*$   satisfying  \eqref{eq:mv_optimal_state},
and $\bu_t^* =a^*(t,  {\bX}^{*}_t, \sE[{\bX}^{*}_t])$. 
This finishes the proof.
\end{proof}

Theorem \ref{thm:potential_maximizer} constructs NEs for portfolio selection games with general heterogeneous preference parameters. The condition for the existence of these equilibria depends only on the interaction weights $(\lambda^V_{ij})_{i,j\in [N]}$ in  the variance and is independent of interactions through the mean. 
In the special case with mean-field dependence $\lambda^V_{ij}=\theta_i/(N-1)$, as analyzed in  \cite{shao2025competitive, huang2025partial}, the condition
$     \sum_{i=1}^N \frac{\varphi_i \theta_i}{1+\varphi_i\theta_i} < 1
$
reduces to
$     \sum_{i=1}^N \frac{  \theta_i}{N-1+ \theta_i} < 1
$, 
which is satisfied if $\theta_i\in (0,1]$ for all $i\in [N]$, with at least one $\theta_i$
strictly less than 1.

  }\section*{Acknowledgments}
XG and YZ  
 are grateful for support from the Imperial Global Connect Fund.
XL is grateful for support from EPSRC grant EP/Y028872/1 (Mathematical Foundations of Intelligence: An Erlangen Programme for AI).
All numerical results are based on simulations conducted by the authors and no external datasets were used.
The authors declare that they have no conflict of interest.
\bibliographystyle{abbrv}
\bibliography{references}

@article{plank2026learning,
  title={Learning Distributed Equilibria in Linear-Quadratic Stochastic Differential Games: An $\alpha $-Potential Approach},
  author={Plank, Philipp and Zhang, Yufei},
  journal={arXiv preprint arXiv:2602.16555},
  year={2026}
}

@article{di2025alpha,
  title={$\alpha $-Potential Games for Decentralized Control of Connected and Automated Vehicles},
  author={Di, Xuan and Hu, Anran and Wang, Zhexin and Zhang, Yufei},
  journal={arXiv preprint arXiv:2512.05712},
  year={2025}
}

@article{huang2025partial,
  title={Partial Information in a Mean-Variance Portfolio Selection Game},
  author={Huang, Yu-Jui and Sun, Li-Hsien},
  journal={Mathematical Finance},
  year={2025},
  publisher={Wiley Online Library}
}

@article{shao2025competitive,
  title={Competitive optimal portfolio selection under mean-variance criterion},
  author={Shao, Guojiang and Xu, Zuo Quan and Zhang, Qi},
  journal={arXiv preprint arXiv:2511.05270},
  year={2025}
}

@article{liu2018optimal,
  title={Optimal inventory control with jump diffusion and nonlinear dynamics in the demand},
  author={Liu, Jingzhen and Cedric Yiu, Ka Fai and Bensoussan, Alain},
  journal={SIAM Journal on Control and Optimization},
  volume={56},
  number={1},
  pages={53--74},
  year={2018},
  publisher={SIAM}
}

@article{musila1991generalized,
  title={Generalized Stein's model for anatomically complex neurons},
  author={Musila, Miroslav and L{\'a}nsk{\`y}, Petr},
  journal={BioSystems},
  volume={25},
  number={3},
  pages={179--191},
  year={1991},
  publisher={Elsevier}
}

@article{sirovich2014cooperative,
  title={Cooperative behavior in a jump diffusion model for a simple network of spiking neurons},
  author={Sirovich, Roberta and Sacerdote, Laura Lea and Villa, Alessandro EP and others},
  journal={Mathematical Biosciences and Engineering},
  volume={11},
  number={2},
  pages={385--401},
  year={2014}
}

@article{han2017deep,
  title={Deep learning-based numerical methods for high-dimensional parabolic partial differential equations and backward stochastic differential equations},
author = {E, Weinan and Han, Jiequn and Jentzen, Arnulf},
  journal={Communications in mathematics and statistics},
  volume={5},
  number={4},
  pages={349--380},
  year={2017},
  publisher={Springer}
}

@article{han2016deep,
  title={Deep learning approximation for stochastic control problems},
  author={Han, Jiequn and E, Weinan},
  journal={arXiv preprint arXiv:1611.07422},
  year={2016}
}

@article{santambrogio2021cucker,
  title={A Cucker--Smale inspired deterministic mean field game with velocity interactions},
  author={Santambrogio, Filippo and Shim, Woojoo},
  journal={SIAM Journal on Control and Optimization},
  volume={59},
  number={6},
  pages={4155--4187},
  year={2021},
  publisher={SIAM}
}

@incollection{tordeux2016collision,
  title={Collision-free speed model for pedestrian dynamics},
  author={Tordeux, Antoine and Chraibi, Mohcine and Seyfried, Armin},
  booktitle={Traffic and Granular Flow'15},
  pages={225--232},
  year={2016},
  publisher={Springer}
}

@article{jakobsen2005continuous,
  title={Continuous dependence estimates for viscosity solutions of integro-PDEs},
  author={Jakobsen, Espen R and Karlsen, Kenneth H},
  journal={Journal of Differential Equations},
  volume={212},
  number={2},
  pages={278--318},
  year={2005},
  publisher={Elsevier}
}

@inproceedings{mazumdar2020policy,
  title={Policy-Gradient Algorithms Have No Guarantees of Convergence in Linear Quadratic Games},
  author={Mazumdar, Eric and Ratliff, Lillian J and Jordan, Michael I and Sastry, S Shankar},
  booktitle={Proceedings of the 19th International Conference on Autonomous Agents and MultiAgent Systems},
  pages={860--868},
  year={2020}
}

@article{giegrich2024convergence,
  title={Convergence of policy gradient methods for finite-horizon exploratory linear-quadratic control problems},
  author={Giegrich, Michael and Reisinger, Christoph and Zhang, Yufei},
  journal={SIAM Journal on Control and Optimization},
  volume={62},
  number={2},
  pages={1060--1092},
  year={2024},
  publisher={SIAM}
}

@article{reisinger2023linear,
  title={Linear convergence of a policy gradient method for some finite horizon continuous time control problems},
  author={Reisinger, Christoph and Stockinger, Wolfgang and Zhang, Yufei},
  journal={SIAM Journal on Control and Optimization},
  volume={61},
  number={6},
  pages={3526--3558},
  year={2023},
  publisher={SIAM}
}

@article{sethi2024entropy,
  title={Entropy annealing for policy mirror descent in continuous time and space},
  author={Sethi, Deven and {\v{S}}i{\v{s}}ka, David and Zhang, Yufei},
  journal={arXiv preprint arXiv:2405.20250},
  year={2024}
}

@article{hu2021deep,
  title={Deep fictitious play for stochastic differential games},
  author={Hu, Ruimeng},
  journal={Communications in Mathematical Sciences},
  volume={19},
  number={2},
  pages={325--353},
  year={2021},
  publisher={International Press of Boston}
}

@misc{kingma2017adammethodstochasticoptimization,
      title={Adam: A Method for Stochastic Optimization}, 
      author={Diederik P. Kingma and Jimmy Ba},
      year={2017},
      eprint={1412.6980},
      archivePrefix={arXiv},
      primaryClass={cs.LG},
      url={https://arxiv.org/abs/1412.6980}, 
}

@article{cont2005finite,
  title={A finite difference scheme for option pricing in jump diffusion and exponential {L}{\'e}vy models},
  author={Cont, Rama and Voltchkova, Ekaterina},
  journal={SIAM Journal on Numerical Analysis},
  volume={43},
  number={4},
  pages={1596--1626},
  year={2005},
  publisher={SIAM}
}

@article{reisinger2021penalty,
  title={A penalty scheme and policy iteration for nonlocal {HJB} variational inequalities with monotone nonlinearities},
  author={Reisinger, Christoph and Zhang, Yufei},
  journal={Computers \& Mathematics with Applications},
  volume={93},
  pages={199--213},
  year={2021},
  publisher={Elsevier}
}

@article{dumitrescu2021approximation,
  title={Approximation schemes for mixed optimal stopping and control problems with nonlinear expectations and jumps},
  author={Dumitrescu, Roxana and Reisinger, Christoph and Zhang, Yufei},
  journal={Applied Mathematics \& Optimization},
  volume={83},
  number={3},
  pages={1387--1429},
  year={2021},
  publisher={Springer}
}

@article{li2018mean,
  title={Mean-field forward and backward {SDEs} with jumps and associated nonlocal quasi-linear integral-{PDEs}},
  author={Li, Juan},
  journal={Stochastic Processes and their Applications},
  volume={128},
  number={9},
  pages={3118--3180},
  year={2018},
  publisher={Elsevier}
}

@Inbook{kunita2004stochastic,
author="Kunita, Hiroshi",
editor="Rao, M. M.",
title="Stochastic Differential Equations Based on L{\'e}vy Processes and Stochastic Flows of Diffeomorphisms",
bookTitle="Real and Stochastic Analysis: New Perspectives",
year="2004",
publisher="Birkh{\"a}user Boston",
address="Boston, MA",
pages="305--373",
isbn="978-1-4612-2054-1",
doi="10.1007/978-1-4612-2054-1_6",
url="https://doi.org/10.1007/978-1-4612-2054-1_6"
}

@article{nourian2011mean,
  title={Mean field analysis of controlled {Cucker-Smale} type flocking: Linear analysis and perturbation equations},
  author={Nourian, Mojtaba and Caines, Peter E and Malham{\'e}, Roland P},
  journal={IFAC Proceedings Volumes},
  volume={44},
  number={1},
  pages={4471--4476},
  year={2011},
  publisher={Elsevier}
}

@article{kavuncu2021potential,
  title={Potential {iLQR}: A potential-minimizing controller for planning multi-agent interactive trajectories},
  author={Kavuncu, Talha and Yaraneri, Ayberk and Mehr, Negar},
  journal={arXiv preprint arXiv:2107.04926},
  year={2021}
}

@article{sun2024imagined,
  title={Imagined Potential Games: A Framework for Simulating, Learning and Evaluating Interactive Behaviors},
  author={Sun, Lingfeng and Wang, Yixiao and Hung, Pin-Yun and Wang, Changhao and Zhang, Xiang and Xu, Zhuo and Tomizuka, Masayoshi},
  journal={arXiv preprint arXiv:2411.03669},
  year={2024}
}

@article{aurell2022stochastic,
  title={Stochastic graphon games: {II}. the linear-quadratic case},
  author={Aurell, Alexander and Carmona, Ren{\'e} and Lauriere, Mathieu},
  journal={Applied Mathematics \& Optimization},
  volume={85},
  number={3},
  pages={39},
  year={2022},
  publisher={Springer}
}

@article{carmona2023synchronization,
  title={Synchronization in a Kuramoto mean field game},
  author={Carmona, Rene and Cormier, Quentin and Soner, H Mete},
  journal={Communications in Partial Differential Equations},
  volume={48},
  number={9},
  pages={1214--1244},
  year={2023},
  publisher={Taylor \& Francis}
}

@article{guo2023ito,
  title={It{\^o}’s formula for flows of measures on semimartingales},
  author={Guo, Xin and Pham, Huy{\^e}n and Wei, Xiaoli},
  journal={Stochastic Processes and their applications},
  volume={159},
  pages={350--390},
  year={2023},
  publisher={Elsevier}
}

@article{guo2025bsde,
  title={BSDE Approach for $\alpha$-Potential Stochastic Differential Games},
  author={Guo, Xin and Li, Xun and Zhang, Liangquan},
  journal={arXiv preprint arXiv:2507.13256},
  year={2025}
}

@article{maheshwari2024convergence,
  title={Convergence of Decentralized Actor-Critic Algorithm in General-sum Markov Games},
  author={Maheshwari, Chinmay and Wu, Manxi and Sastry, Shankar},
  journal={IEEE Control Systems Letters},
  year={2024},
  publisher={IEEE}
}

@article{kalaria2024alpha,
  title={$\alpha$-RACER: Real-Time Algorithm for Game-Theoretic Motion Planning and Control in Autonomous Racing using Near-Potential Function},
  author={Kalaria, Dvij and Maheshwari, Chinmay and Sastry, Shankar},
  journal={arXiv preprint arXiv:2412.08855},
  year={2024}
}

@book{yong2012stochastic,
  title={Stochastic controls: Hamiltonian systems and HJB equations},
  author={Yong, Jiongmin and Zhou, Xun Yu},
  volume={43},
  year={2012},
  publisher={Springer Science \& Business Media}
}

@article{guo2024alpha,
  title={An $\alpha$-Potential Game Framework for N-Player Dynamic Games},
  author={Guo, Xin and Li, Xinyu and Zhang, Yufei},
  journal={SIAM Journal on Control and Optimization},
  volume={63},
  number={4},
  pages={2964--3005},
  year={2025},
  publisher={SIAM}
}

@article{lachapelle2011mean,
  title={On a mean field game approach modeling congestion and aversion in pedestrian crowds},
  author={Lachapelle, Aim{\'e} and Wolfram, Marie-Therese},
  journal={Transportation research part B: methodological},
  volume={45},
  number={10},
  pages={1572--1589},
  year={2011},
  publisher={Elsevier}
}

@article{aurell2018mean,
  title={Mean-field type modeling of nonlocal crowd aversion in pedestrian crowd dynamics},
  author={Aurell, Alexander and Djehiche, Boualem},
  journal={SIAM Journal on Control and Optimization},
  volume={56},
  number={1},
  pages={434--455},
  year={2018},
  publisher={SIAM}
}

@article{jackson2023approximately,
  title={Approximately optimal distributed stochastic controls beyond the mean field setting},
  author={Jackson, Joe and Lacker, Daniel},
  journal={arXiv preprint arXiv:2301.02901},
  year={2023}
}

@inproceedings{aghajani2015formation,
  title={Formation control of multi-vehicle systems using cooperative game theory},
  author={Aghajani, Amin and Doustmohammadi, Ali},
  booktitle={2015 15th International Conference on Control, Automation and Systems (ICCAS)},
  pages={704--709},
  year={2015},
  organization={IEEE}
}

@article{sun2023distributed,
  title={Distributed Multi-agent Interaction Generation with Imagined Potential Games},
  author={Sun, Lingfeng and Hung, Pin-Yun and Wang, Changhao and Tomizuka, Masayoshi and Xu, Zhuo},
  journal={arXiv preprint arXiv:2310.01614},
  year={2023}
}

@article{blum2010routing,
  title={Routing without regret: On convergence to Nash equilibria of regret-minimizing algorithms in routing games},
  author={Blum, Avrim and Even-Dar, Eyal and Ligett, Katrina},
  journal={Theory of Computing},
  volume={6},
  number={1},
  pages={179--199},
  year={2010},
  publisher={Theory of Computing Exchange}
}

@inproceedings{krichene2015convergence,
  title={Convergence of heterogeneous distributed learning in stochastic routing games},
  author={Krichene, Syrine and Krichene, Walid and Dong, Roy and Bayen, Alexandre},
  booktitle={2015 53rd Annual Allerton Conference on Communication, Control, and Computing (Allerton)},
  pages={480--487},
  year={2015},
  organization={IEEE}
}

@inproceedings{colombo2012efficient,
  title={Efficient algorithms for collision avoidance at intersections},
  author={Colombo, Alessandro and Del Vecchio, Domitilla},
  booktitle={Proceedings of the 15th ACM international conference on Hybrid Systems: Computation and Control},
  pages={145--154},
  year={2012}
}

@inproceedings{narasimha2022multi,
  title={Multi-agent learning via markov potential games in marketplaces for distributed energy resources},
  author={Narasimha, Dheeraj and Lee, Kiyeob and Kalathil, Dileep and Shakkottai, Srinivas},
  booktitle={2022 IEEE 61st Conference on Decision and Control (CDC)},
  pages={6350--6357},
  year={2022},
  organization={IEEE}
}

@article{ma2011decentralized,
  title={Decentralized charging control of large populations of plug-in electric vehicles},
  author={Ma, Zhongjing and Callaway, Duncan S and Hiskens, Ian A},
  journal={IEEE Transactions on control systems technology},
  volume={21},
  number={1},
  pages={67--78},
  year={2011},
  publisher={IEEE}
}

@inproceedings{paccagnan2016aggregative,
  title={On aggregative and mean field games with applications to electricity markets},
  author={Paccagnan, Dario and Kamgarpour, Maryam and Lygeros, John},
  booktitle={2016 European Control Conference (ECC)},
  pages={196--201},
  year={2016},
  organization={IEEE}
}

@article{srikantha2016resilient,
  title={Resilient distributed real-time demand response via population games},
  author={Srikantha, Pirathayini and Kundur, Deepa},
  journal={IEEE Transactions on Smart Grid},
  volume={8},
  number={6},
  pages={2532--2543},
  year={2016},
  publisher={IEEE}
}

@inproceedings{ramchurn2011agent,
  title={Agent-Based Control for Decentralized Demand Side Management in the Smart Grid},
    author={Ramchurn, Sarvapali D and Vytelingum, Perukrishnen and Rogers, Alex and Jennings, Nick},
  booktitle={Proc. of the 10th International Joint Conference on Autonomous Agents and Multi-agent Systems, AAMAS 2011},
  pages={5--12},
  year={2011}
}

@article{tushar2020peer,
  title={Peer-to-peer trading in electricity networks: An overview},
  author={Tushar, Wayes and Saha, Tapan Kumar and Yuen, Chau and Smith, David and Poor, H Vincent},
  journal={IEEE transactions on smart grid},
  volume={11},
  number={4},
  pages={3185--3200},
  year={2020},
  publisher={IEEE}
}

@article{lu2025multiagent,
  title={Multiagent Relative Investment Games in a Jump Diffusion Market with Deep Reinforcement Learning Algorithm},
  author={Lu, Liwei and Hu, Ruimeng and Yang, Xu and Zhu, Yi},
  journal={SIAM Journal on Financial Mathematics},
  volume={16},
  number={2},
  pages={707--746},
  year={2025},
  publisher={SIAM}
}

@article{guo2023markov,
  title={Markov $ \alpha $-Potential Games: Equilibrium Approximation and Regret Analysis},
  author={Guo, Xin and Li, Xinyu and Maheshwari, Chinmay and Sastry, Shankar and Wu, Manxi},
  journal={arXiv preprint arXiv:2305.12553},
  year={2023}
}

@article{dumitrescu2017mixed,
  title={Mixed generalized {Dynkin} game and stochastic control in a {Markovian} framework},
  author={Dumitrescu, Roxana and Quenez, Marie-Claire and Sulem, Agn{\`e}s},
  journal={Stochastics},
  volume={89},
  number={1},
  pages={400--429},
  year={2017},
  publisher={Taylor \& Francis}
}

@article{pham1998optimal,
  title={Optimal stopping of controlled jump diffusion processes: a viscosity solution approach},
  author={Pham, Huy{\^e}n},
  journal={J. Math. Syst. Estimat. Control},
  volume={8},
  number={1},
  pages={1},
  year={1998}
}

@book{carmona2016lectures,
  title={Lectures on BSDEs, stochastic control, and stochastic differential games with financial applications},
  author={Carmona, Ren{\'e}},
  year={2016},
  publisher={SIAM}
}

@article{guo2025towards,
  title={Towards an analytical framework for dynamic potential games},
  author={Guo, Xin and Zhang, Yufei},
  journal={SIAM Journal on Control and Optimization},
  volume={63},
  number={2},
  pages={1213--1242},
  year={2025},
  publisher={SIAM}
}

@book{carmona2018probabilistic,
  title={Probabilistic Theory of {Mean Field Games} with Applications I: {Mean Field FBSDEs}, Control, and Games
},
  author={Carmona, Ren{\'e} and Delarue, Fran{\c{c}}ois},
  volume={83},
  year={2018},
  publisher={Springer}
}

@Misc{amsmath,
  author =	 {{American Mathematical Society}},
  title =	 {User's Guide for the \texttt{amsmath} Package
                  (Version 2.0)},
  url =		 {ftp://ftp.ams.org/pub/tex/doc/amsmath/amsldoc.pdf},
  urldate =	 {2015-07-30},
  year =	 2002}

@book{boyd2004convex,
  title={Convex Optimization},
  author={Boyd, Stephen and Boyd, Stephen P and Vandenberghe, Lieven},
  year={2004},
  publisher={Cambridge University Press}
}

 \begin{appendix}
    
\section{Implementation of Algorithm \ref{alg:PG} for Crowd Motion Games} 

\label{sec:implmentation}

To implement Algorithm~\ref{alg:PG}, we uniformly discretize the time interval $[0, 1]$ into $L = 50$ steps. The batch size $M$, representing the number of simulated trajectories per parameter update, is set to 500.

Before stating the configuration details of policy parameterization, we remark that the algorithm’s hyperparameters have not been optimally tuned and hence the following choices may not represent the optimal combination.

We employ a residual feedforward neural network architecture following \cite{lu2025multiagent}, consisting of an input layer, a sequence of residual blocks, and an output layer. Each residual block has the form
$
x \mapsto \varrho(L_1(\varrho(L_2(x)))) + x
$
where $L_1$ and $L_2$ are fully connected layers with matching input and output dimensions, and $\varrho$ denotes the activation function, chosen here to be the standard ReLU.  

Our neural  network policies comprise  four residual blocks, each with width $d + 10$, where $d=4\times 4+1=17$ is the dimensions of 
 the joint state and sensitivity processes, and also  the time variable.
This neural architecture requires no prior knowledge of the solution’s structure. 
Parameters are optimized using the Adam optimizer, with an initial learning rate of $10^{-3}$. A \textsc{ReduceLROnPlateau} scheduler from PyTorch is employed to automatically reduce the learning rate when the validation loss stagnates.
All experiments are run using the fixed random seed 2025.

All
experiments are conducted on a MacBook Pro with 16GB of memory and a Apple M1 Pro chip.

\end{appendix}

\end{document}